\documentclass[11pt,reqno]{amsart}

\usepackage[text={155mm,245mm},centering]{geometry}   

\usepackage{graphicx}
\usepackage{amssymb}
\usepackage{epstopdf}

\usepackage{latexsym}

\usepackage{amsmath,amsthm}

\usepackage[mathscr]{eucal}
\usepackage[all]{xy}
\usepackage{mathrsfs}

\usepackage[plainpages=false]{hyperref}

\newtheorem{theorem}{Theorem}[section]
\newtheorem{lemma}[theorem]{Lemma}

\newtheorem{corollary}[theorem]{Corollary}
\newtheorem{proposition}[theorem]{Proposition}

\newtheorem{assumption}[theorem]{Assumption}

\numberwithin{equation}{section}

\newcommand{\Sp}{\mathrm{Span}}

\theoremstyle{definition}

\newtheorem{definition}[theorem]{Definition}
\newtheorem{definition-lemma}[theorem]{Definition-Lemma}
\newtheorem{definition-theorem}[theorem]{Definition-Theorem}

\newtheorem{remark}[theorem]{Remark}

\newtheorem*{ack}{Acknowledgements}

\newcommand{\ac}{\textup{!`}}

\def\ldb{\mathopen{\{\!\!\{}}
\def\rdb{\mathclose{\}\!\!\}}}
\def\ldbg{\mathopen{\bigl\{\!\!\bigl\{}}
\def\rdbg{\mathclose{\bigr\}\!\!\bigr\}}}

\newcommand{\hdot}{\;\raisebox{3.2pt}{\text{\circle*{2.5}}}}

\allowdisplaybreaks

\title[Koszul Calabi-Yau algebras]{The derived non-commutative Poisson bracket on\\  Koszul Calabi-Yau algebras}

\author{Xiaojun Chen}
\address{Department of Mathematics, Sichuan University, Chengdu, Sichuan Province 610064 P. R. China}
\email{xjchen@scu.edu.cn}

\author{Alimjon Eshmatov}
\address{Department of Mathematics, University of Western Ontario, London, Ontario N6A5B7 Canada}
\email{aeshmato@uwo.ca}

\author{Farkhod Eshmatov}
\address{Department of Mathematics, Sichuan University, Chengdu, Sichuan Province 610064 P. R. China}
\email{olimjon55@hotmail.com}

\author{Song Yang}
\address{Department of Mathematics, Sichuan University, Chengdu, Sichuan Province 610064 P. R. China}
\email{syang.math@gmail.com}

\date{}

\begin{document}

\begin{abstract}
Let $A$ be a Koszul (or more generally, $N$-Koszul) Calabi-Yau algebra.
Inspired by the works of Kontsevich, Ginzburg and Van den Bergh,
we show that there is
a derived non-commutative Poisson structure on $A$,
which induces a
graded Lie algebra
structure on the cyclic homology
of $A$; moreover, we show that
the Hochschild homology of $A$
is a Lie module over the cyclic homology and
the Connes long exact sequence is
in fact a sequence of Lie modules.
Finally,
we show that the Leibniz-Loday bracket associated to the
derived non-commutative Poisson
structure on $A$ is naturally mapped to
the Gerstenhaber bracket on the Hochschild cohomology
of its Koszul dual algebra and hence on that of $A$ itself.
Relations with some other brackets in literature
are also discussed and several
examples are given in detail.
\end{abstract}

\subjclass[2010]{14A22; 16S38}

\keywords{Noncommutative Poisson structure, Calabi-Yau algebra, cyclic homology}

\thanks{Corresponding author: Xiaojun Chen}

\maketitle

\setcounter{tocdepth}{1}
\tableofcontents


\section{Introduction}\label{Sect_Intro}

The notion of Calabi-Yau algebras is introduced by Ginzburg
(\cite{Ginzburg}), and has been intensively studied in recent years.
They are associative algebras with some additional properties,
and may be viewed as non-commutative generalization of affine Calabi-Yau varieties.
It turns out that they are related to
representation theory, non-commutative symplectic/algebraic
geometry, mirror symmetry, and much more.
For more details, see, for example, \cite{Bocklandt,BS,CBEG,Ginzburg,Keller} and references therein.

In this paper, we study the {\it derived non-commutative Poisson structure} on
Koszul (or more generally, $N$-Koszul in the sense of Berger \cite{Berger})
Calabi-Yau algebras, continuing the work of
Berest, Chen, Eshmatov and Ramadoss \cite{BCER}.
Let us start with some backgrounds.

\subsection{Derived non-commutative Poisson structures}

Let $k$ be an algebraically closed field of characteristic zero.
In 2005 Crawley-Boevey (\cite{CB}) introduced for associative algebras
the notion of $\mathrm H_0$-Poisson structure.
Suppose $A$ is an associative algebra over $k$, then an $\mathrm H_0$-Poisson structure on
$A$ is a Lie bracket on $A/[A,A]$ such that
$$[\overline a, -]: A/[A,A]\to A/[A,A]$$
is introduced
by a derivation $d_a: A\to A$, for all $\bar a\in A/[A,A]$.
Such a notion perfectly fits a principle raised by Kontsevich and Rosenberg (\cite{KR})
in the study of non-commutative geometry, that is, any non-commutative geometric
structure (such as the non-commutative symplectic, non-commutative Poisson, etc.)
on a non-commutative space (here we mean an associative algebra) should induce
its classical counterpart on the moduli space of its representations, {\it i.e.} on its representation
scheme.
Recall that for an associative algebra $A$,
$A/[A,A]$ is always considered as the space of functions on $A$, and there
is the canonical {\it trace map}
\begin{equation}\label{trace_map}
\begin{array}{cccl}
\mathrm{Tr}:&A/[A, A]&\longrightarrow& k[\mathrm{Rep}_V(A)]\\
&\bar a&\longmapsto&\{\rho\mapsto \mathrm{trace}(\rho(a))\}
\end{array}
\end{equation}
from the functions on $A$ to the
functions on the representation scheme of $A$ in a $k$-vector space
$V$.
The trace map is $\mathrm{GL}(V)$-invariant,
and
Crawley-Boevey showed that if $A$ admits an $\mathrm H_0$-Poisson structure,
then it naturally induces via the trace map a unique Poisson structure
on $\mathrm{Rep}_V(A)/\!/\mathrm{GL}(V)$ for all $n\in\mathbb N$ such that $\mathrm{Tr}$ is a map of Lie algebras.
The notation ``$\mathrm H_0$"
means the {\it zero-th homology}, since $A/[A, A]$ is the zero-th
Hochschild/cyclic homology of $A$.

In 2012 the $\mathrm H_0$-Poisson structure was generalized
to the higher degree case in \cite{BCER}, where all cyclic homology groups are taken into account.
The starting point is that
the trace map (\ref{trace_map}) is not perfect in the sense that
in very rare cases $\mathrm{Rep}_V(A)$ is a smooth variety (cf. \cite{Cuntz_Quillen1}
for further studies),
and
there are obstructions for $\mathrm{Rep}_V(A)$ to have the desired geometric property.
The work \cite{BKR} shows that instead one has to consider the homotopy category (in the sense of Quillen)
of DG associative algebras.
The main idea is to replace the associative algebra $A$ by its cofibrant resolution $QA$,
and then consider the DG representations of $QA$. It turns out that: (i) there is a surjective map
from the cyclic homology $\mathrm{HC}_\bullet(A)$ to the
homology of the commutator quotient space
$QA/[QA, QA]$; (ii) the
$n$-dimensional DG representation scheme of $QA$, which up to homotopy is denoted by $\mathrm{DRep}_V(A)$,
is smooth in the differential graded sense.
By passing to the homotopy category one obtains a natural
map (called the {\it derived trace map})
\begin{equation}\label{DTM}
\mathrm{HC}_\bullet(A)\longrightarrow \mathrm H_\bullet(\mathrm{DRep}_V(A))
\end{equation}
from the cyclic homology of $A$ to
the homology of the {\it derived representation scheme} of $A$.
For more
details of $\mathrm{DRep}_V(A)$, one may refer to \cite{BFR,BKR}.

The work \cite{BCER} may be viewed as an application of the general result of \cite{BFR,BKR}. In
that paper, an algebra $A$ is called to admit a {\it derived non-commutative Poisson} structure
if there is a DG non-commutative Poisson structure in the sense of Crawley-Boevey
on its cofibrant resolution. It is proved in \cite{BCER} that if $A$ admits
a derived non-commutative Poisson structure,
then (\ref{DTM}) induces a unique graded Poisson structure on the derived representation schemes.

As an important example, it is shown in \cite{BCER}
that the cobar construction $\mathbf \Omega(C)$ of a cyclic coalgebra $C$ (``cyclic" here means
the dual space of $C$ is a cyclic associative algebra)
admits a derived non-commutative Poisson structure.
Since $\mathbf \Omega(C)$ is always a quasi-free (and hence cofibrant) DG algebra,
from the above argument one obtains that
any algebra $A$ which is quasi-isomorphic to $\mathbf \Omega(C)$
admits a derived non-commutative Poisson structure as well.
It is exactly at this point that Koszul Calabi-Yau algebras come in.

\subsection{Koszul Calabi-Yau algebras}
According to Ginzburg \cite{Ginzburg},
an associative algebra $A$ is called
{\it Calabi-Yau of dimension $d$} (or {\it $d$-Calabi-Yau} for short)
if
\begin{itemize}
\item[$-$] $A$ is homologically smooth, that is, it has a finite resolution of
finitely generated projective $A\otimes A^{\mathrm{op}}$ modules;
\item[$-$] there exists an isomorphism
$$
\mathrm{RHom}_{A\otimes A^{\mathrm{op}}}(A, A\otimes A)\simeq A[-d]
$$
in the derived category of $A\otimes A^{\mathrm{op}}$ modules.
\end{itemize}
Ginzburg in {\it loc. cit.} also showed that if a Calabi-Yau algebra $A$ is Koszul,
then its Koszul dual algebra, denoted by $A^!$, is cyclic. Dually, the Koszul dual coalgebra
of $A$, denoted by $A^{\ac}$, is a cyclic coalgebra. From Koszul duality theory there is a quasi-isomorphism
$$\mathbf\Omega(A^{\ac})  \stackrel{\simeq}{\twoheadrightarrow} A,$$
and since $ \mathbf\Omega(A^{\ac})$ is cofibrant, by the work \cite{BCER} sketched above,
we thus obtain a derived non-commutative Poisson structure on $A$.

In fact, a slightly more general class of Calabi-Yau algebras has the above property.
In \cite{Berger} Berger introduced the notion of $N$-Koszul algebras, where $2$-Koszul is Koszul in the usual sense.
If a Calabi-Yau algebra $A$ is $N$-Koszul, then
is Koszul dual coalgebra $A^{\ac}$ is a cyclic $A_\infty$ coalgebra. Denote the cobar construction of
$A^{\ac}$ by $\mathbf\Omega_{\infty}(A^{\ac})$; we still
have $A\simeq\mathbf\Omega_{\infty}(A^{\ac})$.
There are many examples of $N$-Koszul Calabi-Yau algebras, such as
the Sklyanin algebras,
universal enveloping algebra of semi-simple Lie algebras, and Yang-Mills algebras, etc.
The theorem below studies the derived non-commutative Poisson structure on $N$-Koszul Calabi-Yau algebras
in this general setting:

\begin{theorem}
\label{maintheorem1}
Let $A$ be an $N$-Koszul $d$-Calabi-Yau algebra. Then
\begin{enumerate}
\item[(1)] there is a degree $2-d$ derived non-commutative Poisson structure
 on $A$, which induces a degree $2-d$ graded Lie algebra structure on the cyclic homology
 $\mathrm{HC}_{\bullet}(A)$ of $A$; and
\item[(2)] there is a degree $2-d$ Lie module structure on
 the Hochschild homology $\mathrm{HH}_\bullet(A)$
 over $\mathrm{HC}_{\bullet}(A)$.
\end{enumerate}
\end{theorem}

The first statement in the theorem may be viewed as an application of \cite[Lemma 11]{BCER}
to the Koszul Calabi-Yau case, and it also answers a question raised in the last paragraph in {\it loc. cit.} \S5.4;
the second statement is new.

\subsection{The Connes long exact sequence of Lie modules}

The key ingredient in the above theorem is the structure of a differential graded version of the
{\it double Poisson bracket}
in the sense of Van den Bergh \cite{VdB2} on $\mathbf\Omega_{\infty}(\tilde A^{\ac})$,
where $\tilde A^{\ac}:=k\oplus A^{\ac}$ is the co-augmentation of $A^{\ac}$.
According to Van den Bergh,
a double Poisson bracket on an associative algebra, say $R$, is a bilinear map
$$\ldb-,-\rdb: R\times R\to R\otimes R$$
satisfying some additional conditions.
If $R$ admits a double Poisson structure,
then the commutator quotient space $R_{\natural}=R/[R,R]$
naturally admits a Lie algebra structure satisfying the criterion
of Crawley-Boevey.
From the famous result of Feigin and Tsygan \cite{FT},
for a Koszul algebra $A$, the homology of
$\mathbf{\Omega}_\infty(\tilde A^{\ac})_{\natural}$
is exactly $\mathrm{HC}_\bullet(A)$ and Theorem \ref{maintheorem1} part (1) follows.

Moreover,
we recall that for an associative algebra $A$ there is a well-known long exact sequence
due to Connes, relating the cyclic and Hochschild homologies:
$$
\cdots\stackrel{B}\longrightarrow\mathrm{HH}_\bullet(A)\stackrel{I}\longrightarrow
\mathrm{HC}_\bullet(A)\stackrel{S}\longrightarrow\mathrm{HC}_{\bullet-2}(A)\stackrel{B}\longrightarrow
\mathrm{HH}_{\bullet-1}(A)\stackrel{I}\longrightarrow\cdots
$$
By using an elegant interpretation of the Hochschild and cyclic homology of an algebra
in terms of its bar construction, which is due to Quillen \cite{Quillen1},
and by Theorem~\ref{maintheorem1} part (1),
we in fact obtain that $\mathrm{HH}_\bullet(A)$ and $\mathrm{HC}_{\bullet}(A)$
are Lie modules over the Lie algebra $ \mathrm{HC}_{\bullet}(A)$, from which
Theorem \ref{maintheorem1} part (2) follows, and we now claim
that
\begin{theorem}
\label{maintheorem2}
The maps $B,I$ and $S$ are morphisms of Lie modules of degree $2-d$.
\end{theorem}

\subsection{The Poisson structure on derived representation schemes}

The significance of the above two theorems is that the Lie module morphism naturally
induces a Lie module morphism on the (derived) representation scheme of the Calabi-Yau algebra.
As we mentioned above, there is a derived trace map
$$\mathrm{Tr}: \mathrm{HC}_\bullet(A)\to\mathrm H_\bullet(\mathrm{DRep}_V(A)).$$
The images are $\mathrm{GL}(V)$-invariant, and in \cite[Theorem 5.2]{BKR}
this map is extended to Hochschild homology and there is in fact a commutative diagram
\begin{equation}\label{liemod_diag}
\xymatrixcolsep{4pc}
\xymatrix{
\mathrm{HC}_\bullet(A)\ar[r]^B \ar[d]^{\mathrm{Tr}}&\mathrm{HH}_{\bullet+1}(A)\ar[d]^{\mathrm{Tr}}\\
\mathrm{H}_\bullet(\mathrm{DRep}_V(A)^{\mathrm{GL}(V)})\ar[r]^{B_V}&\mathrm H_\bullet(\Omega^1(\mathrm{DRep}_V(A)^{\mathrm{GL}(V)})),
}
\end{equation}
where $B$ in the upper line is the Connes differential, $B_V$ in the bottom line
is the de Rham differential and $(-)^{\mathrm{GL}(V)}$ means the $\mathrm{GL}(V)$-invariant space.

Now recall that in classical Poisson geometry, if a manifold $M$ has a Poisson structure,
then the space of differential forms $\Omega^*(M)$ on $M$ is a Lie module over the space of functions
$\Omega^0(M)$, where the Lie action is given as follows:
for $f\in\Omega^0(M)$, $\omega\in\Omega^*(M)$,
$$[f, \omega]:=\mathscr L_{X_f}\omega,$$
where $\mathscr L_{X_f}$ is the Lie derivative of the vector field $X_f$ associated to $f$.
In derived non-commutative geometry, we have similar results, and claim that

\begin{theorem}[to appear in \cite{CEE}]
Let $A$ be a Koszul Calabi-Yau algebra.
Then the diagram \eqref{liemod_diag}
is a commutative diagram of Lie module morphisms.
\end{theorem}

To prove this theorem, we shall have to discuss the Lie derivative on
the derived representation schemes, which is very much involved. We decide to
give a complete proof in a separate paper; however, for reader's convenience
we give enough backgrounds
in \S\ref{Sect_DRep}.

\subsection{Relations to the Gerstenhaber and the de V\"{o}lcsey-Van den Bergh brackets}

The construction of the double Poisson bracket on $\mathbf \Omega_{\infty}(\tilde A^{\ac})$ is
also
inspired by the Kontsevich bracket in non-commutative symplectic geometry
(see Kontsevich \cite{Kontsevich} as well as Ginzburg \cite{Ginzburg0} and Van den Bergh \cite{VdB2}
for further discussions).
It is direct to check that
if $\ldb-,-\rdb$ is a double Poisson bracket on $R$
then
$$\{-,-\}=\mu\circ\ldb-,-\rdb: R\times R\to R$$
defines a Leibniz-Loday bracket on $R$,
where $\mu$ is the multiplication. For more details
of the Leibniz-Loday bracket, see \S\ref{Sect_double_poisson_alg}.
It has been interesting for a long time
to explore the relationships among the brackets such as
the Gerstenhaber bracket, the Leibniz-Loday bracket and
the non-commutative Poisson bracket of Kontsevich and Van den Bergh, etc.
In this paper we also study this problem in the case of Koszul
Calabi-Yau algebras with some detail.

First, we observe that there is a natural quasi-isomorphism
(a version of non-commutative Poincar\'e duality originally due to Tradler \cite{Tradler})
$$\Phi: \mathrm{CH}_\bullet(A^{\ac})\longrightarrow\mathrm{CH}^\bullet(A^!)$$
from the Hochschild chain complex of $A^{\ac}$ to the Hochschild cochain complex of
$A^!$. Second, we observe that
$\mathbf \Omega(A^{\ac})$ naturally embeds into $\mathrm{CH}_\bullet(A^{\ac})$ as chain complexes.
By combining these two observations, we obtain the following theorem:

\begin{theorem}
\label{maintheorem_DNCP}
Let $A$ be an $N$-Koszul $d$-Calabi-Yau algebra.
Denote by $A^{\ac}$ its Koszul dual coalgebra.
Denote by $\{-,-\}_{\mathrm{DNCP}}$  the Leibniz-Loday bracket
associated to the derived non-commutative Poisson structure on $A$
and by $\{-,-\}_\mathrm{G}$ the Gerstenhaber bracket on $\mathrm{CH}^\bullet(A^!)$, respectively.
Then we have
$$\Phi\circ B\{u,v\}_{\mathrm{DNCP}}= \{\Phi\circ B(u), \Phi\circ B(v)\}_\mathrm{G},$$
for any $u, v\in\mathbf\Omega_{\infty}(A^{\ac})$, where $B$ is the Connes cyclic operator.
\end{theorem}

A classical result of Keller \cite[Theorem 3.5]{Keller0} says that if $A$ is Koszul, then
the Hochschild cohomology of $A$
and of $A^!$ are isomorphic as Gerstenhaber algebras.
His theorem holds for $N$-Koszul algebras, too.
Therefore, Theorem \ref{maintheorem_DNCP} implies that
$\Phi\circ B$, composed with the isomorphism of Keller, maps
the Leibniz-Loday bracket of $A$ to the Gerstenhaber
bracket on the Hochschild cohomology of $A$ itself.

The above result allows us to  give an explicit formula for
the Lie bracket on the cyclic homology and the Lie
module structure on the Hochschild homology.
First, we recall that $(\mathrm{HH}^{\bullet}(A), \cup) $
is a graded commutative algebra and  $(\mathrm{HH}_{\bullet}(A), \cap) $
is a graded module
$$ \cap: \mathrm{HH}_{m}(A)\otimes \mathrm{HH}^n(A) \to \mathrm{HH}_{m-n}(A)\, ,
\, \alpha  \otimes f \mapsto \alpha  \cap f $$
and we denote $\iota_f(\alpha):=\alpha \cap f$.
Now if $A$ is $d$-Calabi-Yau, then there in fact exists
an element $\omega\in\mathrm{HH}_d(A)$ such that
the cap product with $\omega$:
\begin{eqnarray*}
\Psi: \mathrm{HH}^\bullet(A)&\longrightarrow&\mathrm{HH}_{d-\bullet}(A)
\\
f&\longmapsto&   \iota_{f}\omega
\end{eqnarray*}
is an isomorphism, which is called the
{\it non-commutative Van den Bergh-Poincar\'e duality} for $A$.
We can show
\begin{corollary}
\label{corexpbracket}
The Lie bracket of Theorem~\ref{maintheorem1} on $ \mathrm{HC}_{\bullet}(A) $ is given by
\begin{equation}
\label{mainformula}
\{\alpha, \beta\}_{\mathrm{DNCP}}= (-1)^{(d-|\alpha|-1)} \iota_{\Psi^{-1}(B(\alpha))} B(\beta)
\end{equation}
for $\alpha ,\beta \in \mathrm{HC}_{\bullet}(A) $, where
$B: \mathrm{HC}_{\bullet}(A) \to \mathrm{HH}_{\bullet+1}(A)$ is the Connes operator.
The Lie module structure on $\mathrm{HH}_{\bullet}(A)$ is also
given by the bracket \eqref{mainformula} for $\beta \in \mathrm{HH}_{\bullet}(A)$.
\end{corollary}

Another consequence of Theorem~\ref{maintheorem_DNCP} is that it relates
the Lie bracket $\{-,-\}_{\mathrm{DNCP}}$
of Theorem~\ref{maintheorem1}
on $\mathrm{HC}_{\bullet}(A)$
with the Lie bracket $\{-,-\}_{\mathrm{dVV}}$  of de V\"{o}lcsey-Van den Bergh  on
the negative cyclic homology $\mathrm{HC}_{\bullet}^{-}(A)$
introduced in \cite{dTdVVdB}. Let us briefly recall their construction.
Let $\mathrm{CC}_{\bullet}^{-}(A)$ be the negative cyclic complex
 and let $ \pi: \mathrm{CC}_{\bullet}^{-}(A) \to \mathrm{CH}_{\bullet}(A)$
 be the natural projection in to the last two columns (see e.g. \cite[Section 5.1.4.1]{Loday}).
Next, let $\cup$ be the cup product on $\mathrm{HH}^{\bullet}(A)$. Then
\begin{equation}\label{bracket_V}
\begin{array}{cccl}
 \{-,-\}_{\mathrm{dVV}}: &\mathrm{HC}_{n}^{-}(A) \times \mathrm{HC}_{m}^{-}(A)
 &\longrightarrow&
  \mathrm{HC}_{n+m-d+1}^{-}(A)   \\[0.2cm]
& (\eta_1,\eta_2)& \longmapsto &
(-1)^{|\eta_1|+d}B \circ \Psi\big( \Psi^{-1} (\pi(\eta_1)) \cup \Psi^{-1} (\pi(\eta_2))\big)\;\;\;\;\;
\end{array}
\end{equation}
where $B:\mathrm{HH}_q(A) \to \mathrm{HC}_{q+1}^{-}(A)$ is
the map induced from the Connes operator and $\Psi$ is the non-commutative Poincar\'e duality.
From the following commutative diagram (see \cite[Proposition 5.1.5]{Loday})
\begin{equation}
\label{comdiag1}
\xymatrixcolsep{4pc}
  \xymatrix{ \mathrm{HC}_{\bullet} (A)  \ar[r]^{B}\ar[d]^{id} &  \mathrm{HC}_{\bullet+1}^{-}(A) \ar[d]^{\pi} \\
\mathrm{HC}_{\bullet} (A)  \ar[r]^{B}  &  \mathrm{HH}_{\bullet+1}(A) ,}
\end{equation}
we have
\begin{theorem}
\label{maintheorem_dVV}
Let $A$ be an $N$-Koszul $d$-Calabi-Yau algebra. Then
$$  B\{\alpha, \beta\}_{\mathrm{DNCP}} = \{B(\alpha),B(\beta)\}_{\mathrm{dVV}}$$
for any $\alpha, \beta \in  \mathrm{HC}_{\bullet} (A)$.
\end{theorem}

The rest of the paper is devoted to the proof of the above theorems.
It is organized as follows:
in \S\ref{Sect_H_C_of_coalg} we recollect several basic notions such as the Hochschild and cyclic homology
of algebras and coalgebras;
in \S\ref{Sect_double_poisson_alg} we recall the definition of double Poisson algebras in the sense of Van den Bergh
and their bimodules;
in \S\ref{Sect_double_poisson_on_cobar} we show that for a class of coalgebras there is a
double Poisson bracket on their
cobar construction;
in \S\ref{Sect_bimod_of_1_form} we continue to show that, for the cobar construction, when viewing it as
a DG algebra, its non-commutative differential 1-forms admit a double Poisson bimodule structure;
in \S\ref{Sect_N_Kosuzl_CY} we show $N$-Koszul Calabi-Yau algebras have a derived
non-commutative Poisson structure;
in \S\ref{Sect_DRep} we give a brief introduction to derived representation schemes and
derived non-commutative Poisson structures;
in \S\ref{Sect_proof_of_main_thm} we prove the main theorems listed in \S\ref{Sect_Intro};
and in \S\ref{Sect_example_poly_alg} we give by explicit formulas several brackets on
the space of polynomials of several variables.

\begin{ack}
We would like to thank Yuri Berest for helpful communications and
NSFC (No. 11271269) for partial support.
\end{ack}

\section{Hochschild and cyclic homologies of coalgebras}\label{Sect_H_C_of_coalg}

\subsection{Bar and cobar constructions}

Let $k$ be a field of characteristic zero. An associative
$k$-algebra $A$
 is said to be \textit{augmented} if there is
an algebra homomorphism $\epsilon: A \to k$, called
 the {\it augmentation map}. In particular,
$A$ is  canonically isomorphic, as a vector space, to
$k \oplus \bar{A}$, where $\bar{A}=\mathrm{Ker}(\epsilon)$.

The {\it bar construction} of an augmented algebra $A$,
denoted by $\mathbf{B}(A)$,  is a DG coalgebra defined as follows.
First, recall that the {\it suspension}
of a graded vector space $V$ is  the graded vector
space $sV$ such that $(sV)_i=V_{i-1}$. Similarly,
the {\it desuspension} is $s^{-1}V$ such that
$(s^{-1}V)_i=V_{i+1}$.  As a coalgebra  $\mathbf{B}(A)$ is
the tensor coalgebra $T(s\bar{A})$ of the underlying vector
space of $\bar{A}$ located in degree one. The coproduct is
$$
\triangle( a_1, \cdots , a_n) =
\sum_{i=0}^{n} ( a_1,\cdots, a_i) \otimes (a_{i+1},\cdots , a_n)
$$
while the differential is
$$ b'( a_1, \cdots , a_n) = \sum_{i=1}^{n-1}(-1)^{i-1} (a_1, \cdots, a_i \cdot a_{i+1},
\cdots, a_n) \, . $$
The counit $\eta:\mathbf B(A) \to k$ is the projection onto $A^{\otimes 0}=k$ and
we will denote by $\bar{\mathbf {B}}(A)$  the $\mathrm{Ker}(\eta)$.

A coassociative coalgebra $(C,\triangle)$ is a {\it coaugmented} coalgebra
if there is coalgebra homomorphism $u: k \to C$. Then $ C$
can be identified as a vector space to $k \oplus \bar{C}$, where $\bar{C}$
is the $\mathrm{Coker}(u)$. Define the {\it reduced  coproduct}
$\bar{\triangle}: \bar{C} \to \bar{C} \otimes \bar{C}$ as the composite
$$
\bar{C}\stackrel{i}{\longrightarrow}
C\stackrel{\triangle}{\longrightarrow} C \otimes C \stackrel{\pi \otimes \pi}{\longrightarrow}\bar{C} \otimes \bar{C}
$$
where $i$ and $\pi$ are canonical inclusion and projection.
Then {\it cobar construction} of $C$  is a DG algebra $\mathbf{\Omega}(C)$
\footnote{In this paper the cobar construction is denoted by the bold face $\mathbf \Omega$ while the standard
$\Omega$ means the non-commutative differential forms.}
defined as follows. As an algebra
it is  the tensor algebra $T(s^{-1}\bar{C})$ and the differential is
$$  b'(c_1, \cdots, c_n) = \sum_{i=1}^{n}\sum_{(c_i)}
(-1)^{i-1}(c_1,\cdots ,  c_i', c_i'',\cdots, c_n)   $$
where $\bar{\triangle}(c_i)=\sum_{(c_i)} (c_i',c_i'')$.
The unit map $\epsilon : k \to \mathbf{\Omega}(C)$
is the inclusion into $\bar{C}^{\otimes 0}=k$,
and we will write $\overline{\mathbf{\Omega}}(C)$ for $\mathrm{Coker}(\epsilon)$,
which is the {\it reduced cobar construction} of $C$.

\subsection{The cyclic bicomplex }

Let $C$ be a coalgebra over $k$.
We write $\triangle(c)=\sum_{(c)} c'\otimes c'' $ for the coproduct in $C$.
Then we consider the following double complex which is obtained by
reversing the arrows in the standard (Tsygan) double complex of an algebra:
$$
\xymatrixcolsep{4pc}
\xymatrix{
&&&&\\
\ar[r]^{N}
&C^{\otimes 3} \ar[r]^{1-T}\ar[u]^{b}
&C^{\otimes 3} \ar[u]^{b'}\ar[r]^{N}&C^{\otimes 3}
\ar[u]^b\ar[r]^{1-T} &\\
\ar[r]^{N}
&C^{\otimes 2} \ar[r]^{1-T}\ar[u]^{b}
&C^{\otimes 2} \ar[u]^{b'}\ar[r]^{N}&C^{\otimes 2}
\ar[u]^b\ar[r]^{1-T}& \\
\ar[r]^{N}
&C \ar[u]^{b}\ar[r]^{1-T}
&C \ar[u]^{b'} \ar[r]^{N}&C
 \ar[u]^{b} \ar[r]^{1-T}&\\
& \ar[u] 0 &  \ar[u] 0& \ar[u] 0&}
$$
This double complex is $2$-periodic in horizontal direction, with operators
$b', b, T, N$ given by
\begin{eqnarray}
b' ( c_1,\cdots,  c_n)&:=&
\sum_{i=1}^{n}\sum_{(c_i)}
(-1)^{i-1}(c_1,\cdots ,  c_i', c_i'',\cdots, c_n),\nonumber\\[0.1cm]
b ( c_1,\cdots,  c_n)&:= & b' ( c_1,\cdots,  c_n) +   \sum (-1)^n (c_1'',c_2, \cdots, c_n, c_1')\, , \nonumber \\[0.1cm]
T (c_1,\cdots, c_n)&:=& (-1)^{n-1} (c_2, \cdots , c_n, c_1), \nonumber \\[0.1cm]
N &:= & \sum_{i=0}^{n-1} T^{i}.   \nonumber
\end{eqnarray}
The $b$-column is called the Hochschild chain complex $\mathrm{CH}_{\bullet}(C)$
of $C$: it defines the {\it Hochschild homology} $\mathrm{HH}_{\bullet}(C)$.
The $b'$-column is the reduced cobar construction of co-augmented coalgebra
$\tilde{C}=k\oplus C$ shifted by degree one; here $\tilde C$ is the co-augmentation of $C$,
which, as a vector space, is $k\oplus C$, with $k$ being the co-unit.
The kernel of $1-T$ from the $b$-complex to the $b'$-complex is called the cyclic
complex $\mathrm{CC}_{\bullet}(C)$: by definition, its homology is the {\it cyclic homology}
$\mathrm{HC}_{\bullet}$ of $C$.

In practice, to compute the Hochschild and cyclic homologies,
one usually considers the {\it normalized} Hochschild complex.
For a co-augmented coalgebra $C$, the normalized Hochschild complex
$$
\overline{\mathrm{CH}}_n(C):=C\otimes (\bar C)^{\otimes n}
$$
with the differential induced from $b$.
Similarly to the algebra case,
$\overline{\mathrm{CH}}_\bullet(C)$ and $\mathrm{CH}_\bullet(C)$
are quasi-isomorphic; however,
$\overline{\mathrm{CH}}_\bullet(C)$ may be viewed
as a tensor product
$C\otimes\mathbf{\Omega}(C)$
with a {\it twisted} differential given by $id\otimes b'+\tau_{\mathrm L}+\tau_{\mathrm R}$,
where $b'$ is the differential in the cobar construction,
and $\tau_{\mathrm L}$ and $\tau_{\mathrm R}$ are given by
\begin{eqnarray*}
\tau_{\mathrm L}(c_0,c_1,\cdots, c_n)&:=& \sum_{(c_0)}(c_0', c_0'',c_1,\cdots, c_n),\\
\tau_{\mathrm R}(c_0,c_1,\cdots, c_n)&:=&\sum_{(c_0)}(-1)^n(c_0'',c_1,\cdots,c_n,c_0').
\end{eqnarray*}
Under this identification, one sees that the cobar construction embeds
into the Hochschild complex as complexes
\begin{equation}\label{CobartoHoch}
\begin{array}{ccl}
\mathbf{\Omega}(C) &\longrightarrow& \mathrm{CH}_\bullet(C)\\
(c_1,\cdots,c_n ) &\longmapsto& 1\otimes (c_1,\cdots, c_n).
\end{array}
\end{equation}

We now recall some facts about the cyclic bicomplex from Quillen \cite[Section $1.3$]{Quillen1}.
Let $A$ be an associative algebra.
The commutator subspace of $A$ is $[A,A]$
which is the image of $\mu - \mu \sigma: A \otimes A \to A$ where $\mu$
is the product  in $A$ and $\sigma$ is the switching operator, and the commutator quotient space is
$$
A_{\natural}:= A/[A,A] = \mathrm{Coker} ( \mu - \mu \sigma).
$$
Dually, for a coassociative coalgebra $C$ the cocommutator subspace of $C$ is
$$
C^{\natural} := \mathrm{Ker} \{ \triangle - \sigma \triangle : C \to C \otimes C \}.
$$

Let $\tilde A=k \oplus A$ be the augmentation of $A$ (recall
that the augmentation of a not-necessarily unital algebra $A$ is the algebra
$\tilde A=k\oplus A$ where $k$ plays the role of the unit).
Let  $\bar{\mathbf B}$ be the \textit{reduced bar construction} of $\tilde A$.
The following lemma is \cite[Lemma 1.2]{Quillen1}:

\begin{lemma}\label{Hochs_algebra}
The space $\bar{\mathbf{B}}^{\natural}_n$ is the kernel of $1-T$ acting on $A^{\otimes n}$.
Hence, we have the isomorphism of complexes $\mathrm{CC}_{\bullet-1}(A)  \cong \bar{\mathbf{B}}^{\natural}. $
\end{lemma}

\begin{remark}
{\bf For the convenience of later discussions,
from now on we shall shift the degrees of $\mathrm{CC}_\bullet(A)$ up by one, and just write
the above identity and alike as
$\mathrm{CC}_\bullet(A)\cong\bar{\mathbf{B}}^\natural$.}
\end{remark}

Dually, the space $(\mathbf{\Omega}(\tilde C)_{\natural})_n$  is the cokernel of $1-T$
acting on $C^{\otimes n}$. Thus, by the isomorphisms
\begin{equation}
\label{KerCoker}
  \mathrm{CC}_{\bullet}(A)= \mathrm{Coker}(1-T) \cong \mathrm{Ker}(1-T) \, , \,
 \mathrm{CC}_{\bullet}(C) =   \mathrm{Ker}(1-T)  \cong \mathrm{Coker}(1-T),
 \end{equation}
we have the following lemma

\begin{lemma}\label{lemmaident}
As complexes of $k$-vector spaces
$\overline{\mathbf{\Omega}}(\tilde C)_{\natural} \cong \mathrm{CC}_{\bullet}(C)$.
Hence,
$$
\mathrm{HC}_{\bullet}(C) \cong \mathrm{H}_{\bullet}[\overline{\mathbf{\Omega}}(\tilde C)_{\natural}] .
$$
\end{lemma}

There is an efficient way to compute the cyclic homology. Let us recall that (cf. \cite{Loday}) for
a unital and augmented algebra $A=k\oplus \bar A$, its reduced cyclic chain complex
$$\overline{\mathrm{CC}}_\bullet(A):=\mathrm{Coker}\{1-T:\bar C^{\otimes n}\to \bar C^{\otimes n}\}.$$
The associated homology is the reduced cyclic homology of $A$, and is denoted by $
\overline{\mathrm{HC}}_\bullet(A)$,
and there is in fact a decomposition
$$\mathrm{HC}_\bullet(A)\cong\mathrm{HC}_\bullet(k)\oplus\overline{\mathrm{HC}}_\bullet(A).$$
The following is originally due to Feigin and Tsygan, and is now well-known (see e.g. \cite[Proposition 4.2]{BKR}):

\begin{proposition}
\label{propresol1}
Let $R \stackrel{\simeq}{\twoheadrightarrow} A$ be a quasi-free resolution of an algebra $A$.
Then there is quasi-isomorphism $\overline{\mathrm{CC}}_\bullet(A) \to \bar{R}_{\natural}$
inducing isomorphism of homologies $\overline{\mathrm{HC}}_{\bullet}(A) \simeq
\mathrm{H}_\bullet(\bar{R}_{\natural})$, where $\overline{\mathrm{CC}}_\bullet(-)$ and
$\overline{\mathrm{HC}}_\bullet(-)$ are the reduced cyclic chain complex and reduced cyclic homology
respectively.
\end{proposition}

Thus, combining the above two results, one obtains

\begin{corollary}
\label{corimp1}
Let $C$ be a coaugmented coalgebra such that
$\mathbf{\Omega}(C) \stackrel{\simeq}{\twoheadrightarrow} A$
 is a quasi-free resolution of an algebra $A$. Then $\overline{\mathrm{HC}}_{\bullet}(A)
 \cong \overline{\mathrm{HC}}_{\bullet}(C)$. Moreover, if $A$ is an augmented algebra then
 $\mathrm{HC}_{\bullet}(A) \cong \mathrm{HC}_{\bullet}(C)$.
\end{corollary}

\begin{proof}
First, by Proposition~\ref{propresol1},
$\overline{\mathrm{HC}}_{\bullet}(A) \simeq \mathrm{H}_{\bullet}[\overline{\mathbf{\Omega}}(C)_{\natural}]$.
The later by definition is isomorphic to $\mathrm{HC}_{\bullet}(\bar C)$. Since $C$ is coaugmented,
using arguments similar to that in \cite[Proposition 2.2.16]{Loday}, we obtain
$\mathrm{HC}_{\bullet}(\bar C) \cong  \overline{\mathrm{HC}}_{\bullet}(C)$.

The second part follows from the first one, since in this case
$\mathrm{HC}_{\bullet}(A) = \mathrm{HC}_{\bullet}(k) \oplus \overline{\mathrm{HC}}_{\bullet}(A)$
and $\mathrm{HC}_{\bullet}(C) =  \mathrm{HC}_{\bullet}(k) \oplus \overline{\mathrm{HC}}_{\bullet}(C)$.
\end{proof}

\subsection{Non-commutative differential forms on algebras and coalgebras}

The most of the material in this section is either taken from  \cite{Quillen1,Cuntz_Quillen2}
or merely stating the dual version of those results.

For a DG algebra $R$,
we denote by $\Omega^1_R$ the kernel of the multiplication map $R\otimes R\to R$.
It is a DG bimodule induced from outer bimodule structure on $R\otimes R$
and it represents $\underline{\mathrm{Der}}(R,-)$, the complex of $k$-linear graded derivations.
Thus, for any $R$-bimodule $M$, we have
$$
\underline{\mathrm{Der}}(R,M) \, \cong \, \underline{\mathrm{Hom}}_{R^e}(\Omega^1_R, M) \, .
$$
If $M=\Omega^1_R$, the derivation $\partial : R \to \Omega^1_R$ then corresponding to the identity map
under this isomorphism,  is a universal derivation.

 Let $V$ be a $k$-linear vector space and $R$ be the free algebra $\oplus_{n\ge 0} V^{\otimes n}$.  Then
 $\Omega^1_R$ can be identified with $ R \otimes V \otimes R $. Indeed,
 there is a bijection $\mathrm{I} : R \otimes V \otimes R \to \Omega^1_R$ (see \cite[Example $3.10$]{Quillen1})
 \begin{eqnarray}
&& \mathrm{I }  \{ (v_1\, \cdots \,v_{p-1}) \otimes v_p \otimes (v_{p+1}\,  \cdots\, v_m) \}\nonumber\\
 &=& (v_1\, \cdots\, v_p) \otimes (v_{p+1}\, \cdots \,v_m) - (v_1\, \cdots\, v_{p-1})
 \otimes (v_p\, \cdots\, v_m) \label{mapI}
 \end{eqnarray}
  and the universal derivation
 $\partial : R \to  \Omega^1_R$  is given by
  $$  \partial(v_1\,v_2\,\cdots\,v_m)\ = \ \sum_{i=1}^m (v_1\,\cdots\,v_{i-1})
   \otimes v_i \otimes (v_{i+1}\,\cdots\,v_m) \, .$$
  A simple computation shows
\begin{eqnarray}
\mathrm{I} ( \partial(v_1\,v_2\,\cdots\,v_m))
& =& (v_1\,v_2\,\cdots\,v_m) \otimes 1- 1 \otimes (v_1\,v_2\,\cdots\,v_m).\label{Ipartial}
\end{eqnarray}

Next, for an $R$-bimodule $M$,  we define $M_{\natural}:=M/[R,M]$. Then $\Omega^1_{R,\natural}$ is isomorphic
to $R\otimes V$ and the map $\bar{\partial}: R \to \Omega^1_{R,\natural}$ induced by $\partial$ is defined as
\begin{equation}
\label{barpartial}
\bar{\partial}(v_1\,v_2\,\cdots\,v_m) \ = \ \sum_{i=1}^m  (-1)^{(|v_1|+\cdots+|v_i|)(|v_{i+1}|+\cdots+|v_m|)}
(v_{i+1}\,...\,v_m\, v_1\,\cdots\,v_{i-1})  \otimes  v_i.
\end{equation}
One can also define a graded map $\beta: \Omega^1(R)_{\natural} \to R$ given by
\begin{equation}\label{betadef}
  \beta \big( (v_1 \, \cdots\, v_m) \otimes v_{m+1}\big) \
  =\ (v_1\, \cdots\, v_m\,v_{m+1}) - (-1)^{|v_{m+1}|(|v_1|+\cdots+|v_m|)}(v_{m+1}\, v_1\, \cdots\, v_m) \, .
\end{equation}
It is easy to check that   $\beta \bar{\partial} = \bar{\partial} \beta =0$ (see \cite[Proposition 3.8]{Quillen1}).

The above maps give  a commutative diagram
\begin{equation}\label{comdiag}
\xymatrixcolsep{4pc}
\xymatrix{\Omega^1_R\ar@{^(->}[r]\ar[d]_{\natural }& R\otimes  R\ar[d]^{\mu \circ \sigma}\\
\Omega^1_{R,\natural} \ar[r]^{\beta}&R}
\end{equation}
where $\sigma:R \otimes R \to R \otimes R,  r\otimes q \mapsto (-1)^{|r||q|} q \otimes r$ the graded switching operator.

 Now let $(C,\triangle)$ be a coalgebra. Then $M$ is a \textit{bicomodule}
 over $C$ is a vector space equipped with left and right coproducts
 $\triangle_l: M \to C \otimes M,\, \triangle_r : M \to M \otimes C$ defining
 left and right comodule structures which commute:
 $(\triangle_l \otimes 1)\circ \triangle_r= (1\otimes \triangle_r)\circ \triangle_l$.
 The cocommutator subspace of $M$ is
 $$  M^{\natural} \, := \, \mathrm{Ker}\{ \triangle_l - \sigma \triangle_r : M \to C \otimes M\} \, .$$

 We let $\Omega^{C}$, the \textit{non-commutative differentials one-forms} on $C$,
 be the bicomodule $\mathrm{Coker}(\triangle)$. The cocommutator subspace we denote by
 $\Omega^{C,\natural}$.
We can introduce the maps $\beta: C \to \Omega^{C,\natural}$ and
$\bar{\partial}: \Omega^{C,\natural} \to C$.

Recall that $\mathbf{B}$ is the bar construction of $\tilde{A}$. The following
is proved in \cite[Theorem 4]{Quillen1}:

\begin{theorem}\label{HH_algebra}
The complex $\Omega^{\mathbf{B}, \natural}$ is canonically isomorphic
to $\mathrm{CH}_{\bullet}(A)$,  the Hochschild complex of $A$.
Under this identification $\beta = 1-T$ and $\bar{\partial}=N$.
\end{theorem}

Dually we can show

\begin{theorem}
\label{cobhoch}
Let $C$ be a coalgebra and let $R=\mathbf{\Omega}(\tilde C)$, where $\tilde{C}:=k \oplus C$.
Then the complex $\Omega^{1}_{R, \natural}$ is isomorphic to $\mathrm{CH}_{\bullet}(C)$.
Under this identification $\beta = 1-T$ and $\bar{\partial}=N$.
\end{theorem}

It has been pointed out in \cite[Remark 5.14]{Quillen1} that the cyclic bicomplex for the algebra $A$
can be identified with the periodic sequence of complexes
\begin{equation}
\xymatrix{\ar[r]^{\bar{\partial}}& \bar{\mathbf B}  \ar[r]^{-\beta}& \Omega^{\mathbf B, \natural} \ar[r]^{\bar{\partial}}
&
\bar{\mathbf B}  \ar[r]^{-\beta}& \, }
\end{equation}
Similarly, in view of Lemma~\ref{lemmaident} and
Theorem~\ref{cobhoch} one has

\begin{proposition}
\label{propbicomp}
The cyclic bicomplex for the coalgebra $C$ can be identified with
\begin{equation}
\label{bicompbij}
\xymatrix{  \ar[r]^{\bar{\partial}}& \Omega^1_{R,\natural}  \ar[r]^{-\beta}&
 \overline{\mathbf{\Omega}}(\tilde C) \ar[r]^{\bar{\partial}}&
 \Omega^1_{R,\natural}  \ar[r]^{-\beta}& } \,
\end{equation}
\end{proposition}

The following result is essentially established in \cite{JM} (see also \cite{CYZ}):

\begin{lemma}
\label{lemmaident1}
Let $C$ be a coaugmented coalgebra such that $\mathbf{\Omega}(C) \stackrel{\simeq}{\twoheadrightarrow} A$
 is a quasi-free resolution of an algebra $A$. Then $\mathrm{HH}_{\bullet}(A) \cong \mathrm{HH}_{\bullet}(C)$.
\end{lemma}

Combining the above lemma with Theorem
\ref{cobhoch} one obtains

\begin{corollary}
\label{corimp2}
If $\mathbf{\Omega}(C) \stackrel{\simeq}{\twoheadrightarrow} A$ is a quasi-free resolution of $A$
and $R=\mathbf{\Omega}(\tilde C)$
then
\begin{equation}
\mathrm{HH}_{\bullet}(A) \cong \mathrm{H}_{\bullet}(\Omega^{1}_{R, \natural}).
\end{equation}
\end{corollary}

\subsection{The $A_\infty$ algebra and coalgebra case}

The advantage of rephrasing the Hochschild and cyclic homology groups in the proceeding subsections
is that they can be easily generalized to the $A_\infty$ algebra and coalgebra case.

\begin{definition}[$A_\infty$ algebra]
Let $A$ be a graded vector space.
An {\it $A_\infty$ algebra structure} on $A$ is a sequence of linear operators
\begin{equation}\label{Ainfty_1}
m_n: A^{\otimes n}\longrightarrow A, \quad n=1,2,\cdots
\end{equation}
of degree $n-2$ such that
\begin{equation}\label{Ainfty_2}
\sum_{r+s+t=n} (-1)^{r+st}m_{r+1+t}( {id}^{\otimes r}\otimes m_{s}\otimes  id^{\otimes t})=0.
\end{equation}
An $A_\infty$ algebra $A$ is called {\it unital},
if there exists a map $k\to A$ which maps $1$ to $\mathbf 1$, such that
$$
\mu_1(\mathbf 1)=0,\;\;\mu_2(a,\mathbf 1)=\mu_2(\mathbf 1,a)=a,\;\;\mbox{and}\;\;
\mu_n(a_1,\cdots, a_{i-1}, \mathbf 1,a_{i+1}, a_n)=0, \;\;\mbox{for all}\;\; n\ge 3,
$$
where $\mu_1,\mu_2,\cdots$ are the $A_\infty$ operators.
It is called {\it unital and augmented} if
furthermore there is a map
$A\to k$ such that
the composition
$$k \to A\to k $$
is the identity. In this case,
$A$ is decomposed into direct sum
$k\cdot\mathbf 1\oplus\overline A$ of $A_\infty$ algebras.
\end{definition}

\begin{definition}[$A_\infty$ coalgebras]
Let $C$ be a graded vector space.
An {\it $A_\infty$ coalgebra structure} on $C$ is a sequence of linear operators
\begin{equation}
\triangle_n: C \longrightarrow C^{\otimes n}, \quad n=1,2,\cdots
\end{equation}
of degree $n-2$ such that
\begin{equation}
\sum_{r+s+t=n} (-1)^{r+st}( {id}^{\otimes r}\otimes\triangle_{s}\otimes id^{\otimes t})\triangle_{r+1+t}=0.
\end{equation}
An $A_\infty$ coalgebra $C$ is said to be {\it co-unital} if there is a map
$\eta: C\to k$ such that for all $c\in C$,
\begin{eqnarray*}
\eta\circ \triangle_1(c)&=&0,\\
(\eta\otimes id+id\otimes \eta)\circ \triangle_2(c)&=&1\otimes c+c\otimes 1,\quad\mbox{and}\\
\Big(\sum_{j+k+1=n}id^{\otimes j}\otimes \eta\otimes id^{\otimes k}\Big)\circ \triangle_n(c)
&=&0, \quad \mbox{for all}\; n\ge 3.
\end{eqnarray*}
It is called {\it co-unital and co-augmented} if furthermore
there is a map
$k\to C$ such that
the composition
$k\to C\to k$ is the identity.
\end{definition}

\begin{assumption}\label{Assumpt_coalg}
From now on we shall assume\footnote{This assumption is also used by Prout\'e in his definition of
$A_\infty$ coalgebras; see \cite[D\'efinition 3.2]{Proute}.} that for an $A_\infty$-coalgebra $C$,
in each grading $C_i$ is finite dimensional (in literature $C$ is called
{\it locally finite dimensional}), and that
all but finitely many $\triangle_n(v)$ vanish, for each $v\in C$.
In particular, the Koszul dual $A_\infty$ coalgebra $A^{\ac}$ of an $N$-Koszul algebra (to be
studied later), which only has $\triangle_2$ and $\triangle_N$, satisfies this assumption.
\end{assumption}

\begin{definition}[Bar and cobar constructions]
Suppose $A$ is a unital and augmented $A_\infty$ algebra, and
$C$ is a co-unital and co-augmented $A_\infty$ coalgebra (satisfying Assumption \ref{Assumpt_coalg} above).
The {\it bar construction} of $A$, denoted by $\mathbf B_\infty(A)$,
is the quasi-free graded coalgebra $T(s \overline A)$ generated by $ s\overline A$, with differential $d=d_1+d_2+\cdots$,
where $d_n$ is
defined on the co-generators by
\begin{eqnarray*}
d_n: (s\overline A)^{\otimes n}\quad &\longrightarrow&\quad s\overline A\subset T(s\overline A)\\
(s\overline a_1,\cdots, s\overline a_n)&\longmapsto&(-1)^{(n-1)|a_1|+(n-2)|a_1|+\cdots+|a_{n-1}|}s\circ\pi\circ
\mu_n(a_1,\cdots,  a_n)
\end{eqnarray*}
and is extended to $T(s\overline A)$ by derivation, where $\pi: A\to \overline A$ is the projection.

Similarly, the {\it cobar construction} of $C$, denoted by $\mathbf\Omega_\infty(C)$,
is the quasi-free graded algebra $T(s^{-1}\overline C)$ generated by $s^{-1}\overline C$
with differential $d$ defined on the generators by
\begin{eqnarray*}
s^{-1}\overline C&\longrightarrow& T(s^{-1}\overline C)\\
s^{-1}\overline c&\longmapsto&(s^{-1}\circ \pi\circ
\triangle_1+(s^{-1})^{\otimes 2}\circ(\pi^{\otimes 2})\circ \triangle_2+\cdots)(c),
\end{eqnarray*}
which extends to $T(s^{-1}(A))$ by derivation, where in the above expression, $\pi: C\to\overline C$ is the projection.
\end{definition}

\begin{definition}[Hochschild and cyclic homology of $A_\infty$ algebras and coalgebras]
(1) For an $A_\infty$ algebra $A$,
define its {\it Hochschild homology} $\mathrm{HH}_\bullet(A)$ to be
$\mathrm{H}_\bullet(\Omega^{\mathbf B_\infty(\tilde A),\natural})$
as in Theorem
\ref{HH_algebra},
and its {\it cyclic homology} $\mathrm{HC}_\bullet(A)$ to be
$\mathrm{H}_\bullet(\bar{\mathbf B}_\infty(\tilde A)^\natural)$ as in Lemma \ref{Hochs_algebra}.

(2) For an $A_\infty$ coalgebra $C$,
define its {\it Hochschild homology} $\mathrm{HH}_\bullet(C)$ to be
$\mathrm H_\bullet(\Omega_{\mathbf{\Omega}_\infty(\tilde C),\natural}^1)$
as in Theorem \ref{cobhoch},
and its {\it cyclic homology} $\mathrm{HC}_\bullet(C)$ to be
$\mathrm H_\bullet(\overline{\mathbf\Omega}_{\infty}(\tilde C)_\natural)$
as in Lemma \ref{lemmaident}.
\end{definition}

In the above definition, $\tilde A$ and $\tilde C$ are the augmentation of $A$
and the co-augmentation of $C$, respectively, which are defined to be the same
as for algebras and coalgebras.

\begin{remark}
As has been shown above,
for DG algebras and coalgebras,
the definitions given above coincide with the standard ones (cf. Loday \cite{Loday}).
\end{remark}

\begin{proposition}Throughout
Lemma \ref{Hochs_algebra}-Corollary \ref{corimp2},
the statements remain true when the
algebra $A$ is replaced by an $A_\infty$ algebra $A$ and respectively the
 coalgebra $C$ is replaced by an $A_\infty$ coalgebra $C$.
\end{proposition}

\begin{proof}
In the proofs of these statements, the only fact that is used is that
$\mathbf B(\tilde A)$ is a quasi-free DG coalgebra and
$\mathbf\Omega(\tilde C)$ is a quasi-free DG algebra, which is also true for
$\mathbf B_\infty(\tilde A)$ and $\mathbf\Omega_\infty(\tilde C)$ respectively.
\end{proof}

Finally, we show that for unital and augmented $A_\infty$ algebras and coalgebras,
their Hochschild chain complex is
quasi-isomorphic to their {\it normalized Hochschild chain complex}.

\begin{proposition}
Suppose $A$ is a unital and augmented $A_\infty$ algebra and $C$ is a co-unital and co-augmented
$A_\infty$ coalgebra. The the following are quasi-isomorphic
$$
\mathrm{CH}_\bullet(A)\simeq\overline{\mathrm{CH}}_\bullet(A),
\quad
\mathrm{CH}_\bullet(C)\simeq\overline{\mathrm{CH}}_\bullet(C).$$
\end{proposition}

\begin{proof}
The same argument for algebras (see Loday \cite[Proposition 1.6.5]{Loday})
remains to hold for the $A_\infty$ case.
\end{proof}

As an application, we see that
$$
\overline{\mathrm{CH}}_\bullet(A)\simeq A\otimes \mathbf B_{\infty}(A),\quad
\overline{\mathrm{CH}}_\bullet(C)\simeq C\otimes\mathbf \Omega_{\infty}(C),
$$
with the differential properly defined. In particular,
there is an embedding
\begin{equation}\label{Cobar_to_Hoch_Ainfty}
\begin{array}{cccl}
 \mathbf\Omega_{\infty}(C)&\longrightarrow&\overline{\mathrm{CH}}_{\bullet}(C)\\
(a_1,\cdots, a_n)&\longmapsto&(1,a_1,\cdots, a_n)
\end{array}
\end{equation}
of chain complexes similar to the coalgebra case (compare with \eqref{CobartoHoch}).

\section{Double Poisson algebras and bimodules}\label{Sect_double_poisson_alg}

In this section we remind the definition of a double Poisson algebra $A$,
and introduce the notion of a double Poisson bimodule $M$.
After that, we discuss the construction of $n$-Poisson structures.

\subsection{Double Poisson algebras}
\begin{definition}
Suppose $A$ is a unital, associative algebra over a field $k$.
A \textit{double bracket} on $A$ is a bilinear map
$\ldb-,-\rdb:A\times A\to A\otimes A $
which  satisfies
\begin{eqnarray}
\label{db}
&& \ldb a,b\rdb=-\ldb b,a\rdb^\circ,  \\
\label{db2}
&& \ldb a,bc \rdb = b \ldb a,c \rdb + \ldb a,b \rdb c \, ,
\end{eqnarray}
where $(u\otimes v)^\circ =v\otimes u$. Here the action of $b$ and $c$ is
given via outer bimodule structure on $A\otimes A$, that is, $b(a_1\otimes a_2)c:=
ba_1\otimes a_2c$.
We recall that the {\it inner} bimodule structure on $A\otimes A$ is given by
\begin{equation}\label{inner_bimoduleact}
b\,  \ast   (a_1 \otimes a_2)  \ast \, c:=  a_1c \otimes b a_2.
\end{equation}
The formulas  \eqref{db} and \eqref{db2} imply that $\ldb-,-\rdb$
is  a derivation on its first argument for the inner bimodule structure
\begin{equation}
\ldb ab , c \rdb = a \ast \ldb b,c \rdb + \ldb a, c \rdb \ast b.
\end{equation}

Suppose that $\ldb-,-\rdb$ is a double bracket on $A$.
For $a,b_1,...,b_n\in A$, let
$$\ldb a, b_1 \otimes \cdots \otimes b_n\rdb_L \,:=\, \ldb a,b_1 \rdb \otimes b_2 \otimes \cdots \otimes b_n,$$
and let
$$\sigma_s(b_1 \otimes \cdots \otimes b_n):=b_{s^{-1}(1)} \otimes \cdots \otimes b_{s^{-1}(n)},$$
where $s$ is a permutation of $\{1,2,\cdots, n\}$.
If furthermore $A$ satisfies the following
\textit{double Jacobi identity}
\begin{equation}\label{dJ}
\ldbg a , \ldb b,c \rdb \rdbg_L + \sigma_{(123)}\ldbg b,\ldb c,a\rdb \rdbg_L + \sigma_{(132)}
\ldbg c,\ldb a,b\rdb \rdbg_L =0,
\end{equation}
then $A$ is called a \textit{double Poisson algebra}.
\end{definition}

Let $\mu:A \otimes A \to A$ denote the multiplication on $A$, and let
$\{-,-\}:=\mu \circ \ldb-,-\rdb: A\otimes A \to A$. Then $\{-,-\}$ induces
 well-defined maps $A_{\natural} \times A \to A$ and
 $A_{\natural} \times A_{\natural} \to A_{\natural} $ (see  \cite[Lemma 2.4.1]{VdB2}).
Futhermore, the latter bracket is anti-symmetric.

\begin{definition}
A left {\it Leibniz-Loday algebra}\footnote{This algebraic structure is introduced by Loday,
which he calls {\it Leibniz algebra},
while some other authors, for example, Van den Bergh \cite{VdB2}, call it {\it Loday algebra};
we here combine these two terminologies together.}
is a vector space $L$ with a bilinear operation $[-,-]$ such that it satisfies
$$ [a, [b,c]]=[[a,b],c]+[b,[a,c]] \, . $$
\end{definition}

From the definition one immediately sees that

\begin{lemma}\label{lemma:LL}
$(A,\{-,-\})$ is a left Leibniz-Loday algebra.
\end{lemma}

As a consequence of Lemma \ref{lemma:LL} and
\eqref{db}, we have

\begin{corollary} \label{cor_VdB}
If $A$ is a double Poisson algebra, then $\{-,-\}$ makes $A_\natural=A/[A,A]$ into a Lie algebra
and $A$ into a Lie module over $A_\natural$.
\end{corollary}

\begin{proof}
See Van den Bergh \cite[Lemmas 2.4.2 and 2.6.2]{VdB2}.
\end{proof}

\subsection{Double Poisson bimodules}

\begin{definition}
\label{Defofdouble}
Let $A$ be a  double Poisson algebra with the double bracket $\ldb-,-\rdb$
and let $M$ be an $A$-bimodule.
Then  a  {\it double Poisson bracket} on $M$ is
a bilinear product $\ldb-,-\rdb_M : A \times M \to  (A\otimes M) \oplus (M \otimes A)$ such that the following axioms
hold for all $a,b \in A$ and all $m,n \in M$:
\begin{enumerate}
\item[$(i)$]
$ \ldb a,bm \rdb_M \ = \ \ldb a,b\rdb \, m + b \, \ldb a,m\rdb_M $ \, , \,
$ \ldb a,m b \rdb_M \ = \ \ldb a,m\rdb_M \, b + m \, \ldb a,b\rdb_M $;
\item[$(ii)$]
$\ldb ab,m\rdb_M \ = \ a \ast \ldb b,m \rdb_M +  \ldb a,m\rdb_M \ast b$.
\end{enumerate}
\end{definition}

\begin{remark}
Expressions $\ldb a,b\rdb \, m$ and $m \, \ldb a,b\rdb_M$ should be understood as follows.
If $\ldb a,b\rdb = \ldb a,b\rdb ' \otimes \ldb a,b\rdb''$,
then $\ldb a,b\rdb \, m \in A\otimes M$ via the left action of $\ldb a,b\rdb''$ on
$m$ and $m \, \ldb a,b\rdb \in M\otimes  A$ via
the right action of $\ldb a,b\rdb ' $ on $m$.
Also, $\ast$ in $\ a \ast \ldb b,m \rdb_M$ and $\ldb a,m\rdb_M \ast b$
is the action of $A$ on $A\otimes M$ and $M\otimes A$
as inner bimodules
(compare to \eqref{inner_bimoduleact}).
\end{remark}

Our next is to introduce the Jacobi identity for the double bracket $\ldb-,-\rdb_M $.
For this we need to define the following expressions:
$\ldbg a , \ldb b,m \rdb_M \rdbg_{L}, \ldbg b,\ldb m,a\rdb_M \rdbg_{L}$,  $ \ldbg m, \ldb b,a\rdb \rdbg_{L}$.
First, we define
\begin{equation}
\label{doublemb}
\ldb m,b \rdb_M  :=  - (\ldb b,m \rdb_M)^{\circ} \, ,
\end{equation}
that is, if $\ldb b,m \rdb_M = (b_1 \otimes m_1) \oplus (m_2 \otimes b_2)$, then
$\ldb m,b \rdb_M  = - (b_2 \otimes m_2) \oplus (m_1 \otimes b_1)$.
Then
 \begin{eqnarray}
 \label{doublemb2}
  \ldbg a , \ldb b,m \rdb_M \rdbg_{L} := \big( \ldbg a, b_1 \rdbg \otimes m_1 \big)
 \oplus \big( \ldbg a, m_2\rdbg_M \otimes b_2 \big) \, ,
 \end{eqnarray}
 which is in $ ( A \otimes A \otimes M) \oplus (A \otimes M \otimes A)$. Using \eqref{doublemb} and
  \eqref{doublemb2} , we can define $\ldbg b,\ldb m,a\rdb_M \rdbg_{L}$.
Finally, if  $\ldb a, b\rdb= \ldb a, b\rdb' \otimes \ldb a, b\rdb''$, then
  $$ \ldbg m, \ldb a, b \rdb \rdbg_{L} :=  \ldbg m, \ldb a, b\rdb' \rdbg_M  \otimes \ldb a, b\rdb'' \in
  (A\otimes M  \otimes A) \oplus (M \otimes A \otimes A) \, .$$

\begin{definition}
\label{Defofdoubleii}
Let $A$ be a  double Poisson algebra with the double bracket $\ldb-,-\rdb$
and let $M$ be an $A$-bimodule with a double bracket $\ldb-,-\rdb_M$.
Then we say that $M$ is a {\it double Poisson $A$-bimodule} if
\begin{equation}
\label{Jacequiv}
 (iii)   \, \quad \ldbg a , \ldb b,m \rdb_M \rdbg_{L} + \sigma_{(123)}\ldbg b,\ldb m,a\rdb_M \rdbg_{L} +
 \sigma_{(132)} \ldbg m, \ldb a, b \rdb \rdbg_{L} =0 \, ,
\end{equation}
for all $a,b \in A$ and $m,n \in M$.
\end{definition}

\begin{remark}
It is clear that $A$ itself a double Poisson $A$-bimodule.
\end{remark}

Let $(M, \ldb-,-\rdb_M)$ be a double Poisson $A$-bimodule.
Then we define
$$ \{-,-\}_M= \mu_M  \circ \ldb -, - \rdb_M \, : \, A \times M \to M \, ,    $$
where $\mu_M: (A \otimes M) \oplus (M\otimes A) \to M $ is the bimodule action map.
And one can prove

\begin{proposition}
$\{- ,-\}_M$ induces well-defined maps $A_{\natural} \times M \to M$
and $A_{\natural} \times M_{\natural} \to M_{\natural}$.
\end{proposition}

\begin{proof}
(1) Let $a,b \in A$ and $m \in  M$. Then
\begin{eqnarray}
\{ab -ba , \,  m\}_M & = &\mu_M\, \ldb ab, \, m \rdb_M  - \mu_M\, \ldb ba, \, m \rdb_M  \nonumber \\
&=& \mu_M\,  ( a \ast \ldb b,\, m \rdb_M  + \ldb a,\, m \rdb_M \ast b) -
\mu_M\, (  b \ast \ldb a,\, m \rdb_M + \ldb b , \, m \rdb_M \ast a) \nonumber  \\
&=& \ldb b,\, m \rdb_M'\,  a \,  \ldb b,\, m \rdb_M''  + \ldb a,\, m \rdb_M' \, b \,  \ldb a,\, m \rdb_M'' \nonumber \\
&-&  \ldb a,\, m \rdb_M' \, b \,  \ldb a,\, m \rdb_M'' -  \ldb b,\, m \rdb_M' \,  a \,  \ldb b,\, m \rdb_M''  =0 \, .\nonumber
\end{eqnarray}

(2)
We let
\begin{equation}
\label{naturalbracket}
\{-,-\}_{M, \natural}:= \natural \circ \{-,-\}_M \,:  A \times M \to M_{\natural}   ,
\end{equation}
where   $\natural : M \to M_{\natural}$ is the projection map.  Then
$$ \{ a, c m -m c \}_{M,\natural}  = \natural \big ( \{a,c\} m- m\{a,c\} \big) +
 \natural  \big( c \{a, m \}_M - \{a, m\}_M c \big) = 0 \, , $$
and hence by $(1)$ we have $\{-,-\}_{M, \natural}: A_{\natural} \times M_{\natural} \to M_{\natural}$.
\end{proof}

Next, we establish the following statement

\begin{proposition}
The brackets $\{-,-\}_M$ and $\{-,-\}_{M,\natural}$ define on $M$ and $M_{\natural}$ respectively Lie module structures
over the Lie algebra $A_{\natural}$.
\end{proposition}

\begin{proof}
We need to show that
\begin{equation}
\label{Lieiden1}
 \{ \{a,b\}, m\}_{M} = \{a,\{b,m\}_M \}_M -\{b,\{a,m\}_M\}_M  \, .
 \end{equation}
Indeed, we have
\begin{eqnarray}
\{ \{a,b\}, m\}_{M} & = & \mu_M \big[ \ldbg \ldb a, b \rdb' \cdot  \ldb a, b \rdb'', m \rdbg_M \big] \nonumber \\[0.1cm]
&=& \mu_M \big[ \ldb a, b \rdb'  \ast   \ldbg  \ldb a, b \rdb'', m \rdbg_M +  \ldbg  \ldb a, b \rdb' , m \rdbg_M \ast
\ldb a, b \rdb'' \big]  \nonumber \\[0.1cm]
&=&  \mu_M \big[  (1\otimes \mu_M) \circ \sigma_{(132)} \ldbg m, \ldb b, a \rdb \rdbg_{L}   -
(\mu_M \otimes 1) \circ \sigma_{(132)} \ldbg m, \ldb a, b \rdb \rdbg_{L} \big], \nonumber\\
\{a,\{b,m\}_M \}_M &=& \mu_M \big[ \ldbg a, \ldb b, m \rdb'_M \cdot \ldb b, m \rdb''_M \rdbg_M \big] \nonumber \\[0.1cm]
&=&   \mu_M \big[ \ldbg a, \ldb b, m \rdb'_M \rdbg \cdot \ldb b, m \rdb''_M +
\ldb b, m \rdb'_M \cdot \ldbg a, \ldb b, m \rdb''_M \rdbg \big] \nonumber \\[0.1cm]
& =&  \mu_M \big[ ( \mu_M\otimes 1) \, \ldbg a , \ldb b,m \rdb_M \rdbg_{L} - (1 \otimes \mu_M) \circ \sigma_{(123)}
\ldbg a, \ldb m,b \rdb_M \rdbg_L \big], \nonumber\\
\{b,\{a,m\}_M\}_M &=& \mu_M \big[ \ldbg b, \ldb a, m \rdb_M' \cdot \ldb a,m \rdb_M'' \rdbg_M \big] \nonumber \\[0.1cm]
&=&  \mu_M \big[ \ldbg b, \ldb a, m \rdb_M'  \rdbg  \cdot  \ldb a,m \rdb_M'' +
\ldb a, m \rdb_M' \cdot  \ldbg b, \ldb a, m \rdb_M'' \rdbg \big] \nonumber \\[0.1cm]
&=&  \mu_M \big[  (1\otimes \mu_M) \ldbg b, \ldb a, m \rdb_M  \rdbg_L - ( \mu_M\otimes 1) \circ
\sigma_{(123)}   \ldbg b, \ldb m,a  \rdb_M  \rdbg_L \big]  \, .\nonumber
\end{eqnarray}
Using these identities and \eqref{Jacequiv} we obtain \eqref{Lieiden1}.
\end{proof}

\subsection{Double $n$-Poisson structures }

One can easily generalize the above notions to the $n$-graded case.
Indeed, let $A=\oplus_{i\in\mathbb{Z}} A_i $  be  a  $\mathbb{Z}$-graded
vector space. We denote by $A[n]$ the graded vector space with degree shifted by $n$, explicitly,
$A[n] = \oplus (A[n])_i$ with $ (A[n])_i=A_{i-n}$.  Then a \textit{double bracket}
of degree $n$ on $A$ is a bilinear map $\ldb-,-\rdb: A \times A \to (A \otimes A)[-n]$
satisfying
\begin{eqnarray}
\label{dbgrn2}
\ldb a,b\rdb &=& -(-1)^{(|a|+n)(|b|+n)}  \ldb b,a\rdb^\circ,  \\
\label{db22}
\ldb a,b \cdot c \rdb &=&  \ldb a,b \rdb \cdot c + (-1)^{|b|(|a|+n)}  b \cdot \ldb a,c \rdb \, .
\end{eqnarray}
The pair $(A,\ldb-,-\rdb)$ is a \textit{double n-Poisson} algebra if in addition it satisfies the $n$-graded
version of the Jacobi identity
$$ \ldbg a, \ldb b, c \rdb \rdbg_L + (-1)^{(|a|+n)(|b|+|c|)} \sigma_{(123)} \ldbg b, \ldb c, a \rdb \rdbg_L
 + (-1)^{(|c|+n) (|a|+|b|)} \sigma_{(132)} \ldbg c, \ldb a, b \rdb \rdbg_L = \, 0.$$
Similarly, we can define the notion of a double $n$-Poisson bimodule and $n$-graded version of
results of previous two sections can be summarized in the following proposition
\begin{proposition}
\label{mainprop}
Let $A$ be a double $n$-Poisson algebra and $M$ be a double $n$-Poisson $A$-bimodule.
Then
\begin{enumerate}
\item[$\mathrm{(i)}$] the bracket $\{-,-\}$ makes $A_{\natural}$ into an $n$-Lie algebra and $A$ into $n$-Lie
module;
\item[$\mathrm{(ii)}$] the brackets $\{-,-\}_M$ and  $\{-,-\}_{M,\natural}$ make $M$ and $M_{\natural}$
respectively into $n$-Lie module over $A_{\natural}$.
\end{enumerate}
\end{proposition}

\section{Double Poisson structure on the cobar construction}\label{Sect_double_poisson_on_cobar}

In this section, we will define a natural double bracket on the DG algebra $R$ which is
the cobar construction of a cyclic $A_\infty$ coalgebra.

\subsection{Cyclic algebras and coalgebras}\label{Subsect_cyclic}

The notion of {\it cyclic $A_\infty$ algebras} is first introduced by Kontsevich (\cite{Kontsevich}).
Recall that a \textit{symmetric bilinear form of degree} $-d$ on a
 graded vector space $V$ is a bilinear pairing  $ \langle \, - , \, - \rangle : V \otimes V
 \to k[d]$  such that
\begin{equation}
\label{bilform1}
\langle v , w \rangle=  (-1)^{|v||w|} \langle w, v \rangle \quad \mbox{for all} \, v, w \in V\, .
  \end{equation}
A cyclic $A_\infty$ algebra $A$ is an $A_\infty$ algebra with a non-degenerate symmetric bilinear form of degree $-d$
$$\langle -,-\rangle:A\otimes A\to k[d] $$
 such that
\begin{equation}\label{Ainfty_pairing}
\langle \mu_n(a_1,a_2,\cdots,a_n), a_{n+1}\rangle
=(-1)^{n+|a_{n+1}|(|a_1|+\cdots+|a_{n}|)}\langle \mu_n(a_{n+1},a_1,\cdots a_{n-1}),a_n\rangle
\end{equation}
for all $n\in\mathbb N$ and all $a_0,a_1,\cdots, a_n\in A$.
Similarly, one can define {\it cyclic $A_\infty$ coalgebras:}

\begin{definition}[Cyclic $A_\infty$ coalgebra]
Suppose $(C,\{\triangle_k\})$ is an $A_\infty$ coalgebra.
$C$ is called {\it cyclic of degree $-d$}
if there is a non-degenerate symmetric bilinear form of degree $-d$
$$\langle-,-\rangle:C\otimes C\to k[d] $$
such that for any $a, b\in C$,
\begin{equation}
\label{cyclic_pairing}
\langle a , b^1\rangle \cdot b^2\cdots b^{r}
=(-1)^{r+|b^1|(|a|+r)}\langle b, a^r\rangle  \cdot a^1 \cdots a^{r-1}\in
C^{\otimes (r-1)},
\end{equation}
where we write $\triangle_r(a)=a^1a^2\cdots a^r$ and $\triangle_r(b)=b^1b^2\cdots b^r$.
\end{definition}

The notion of cyclic $A_\infty$ coalgebras will be useful for our next discussion.
 We have the following lemma:

\begin{lemma}
\label{equivofalgcoal}
Suppose $C$ is of finite dimension. Then $C$ is a cyclic $A_\infty$ coalgebra if and only if
$A:=C^*=\mathrm{Hom}_k(C,k)$
is a cyclic $A_\infty$ algebra.
\end{lemma}

\begin{proof}
Denote the $A_\infty$ operators of $C$ and $A$ by $\triangle_1,\triangle_2,\cdots$ and $\mu_1,\mu_2,\cdots$, respectively.
Under the isomorphism of $k$-vector spaces
$$\begin{array}{ccl}
C&\stackrel{\cong}{\longrightarrow} & A\\
a&\longmapsto & a^*:=\langle-, a\rangle,
\end{array}
$$
consider the evaluation of both sides of (\ref{cyclic_pairing})
on any $(x_1, x_2, \cdots, x_{r-1})\in A^{\otimes (r-1)}$. We have:
\begin{eqnarray}(x_1,x_2,\cdots, x_{r-1})
\big(\langle b,a^r \rangle (a^1\cdots a^{r-1})\big)
&=&\langle b,a^r\rangle\cdot (x_1,x_2,\cdots, x_{r-1})(a^1\cdots a^{r-1})\nonumber\\
&=&(x_1,x_2, \cdots x_{r-1}, b^*)(a^1 a^2 \cdots a^r)\nonumber\\
&=&(x_1,x_2, \cdots, x_{r-1}, b^*)\triangle_r(a)\nonumber\\
&=&(\mu_r(x_1,x_2, \cdots, x_{r-1}, b^*)(a)\nonumber\\
&=&\langle\mu_r(x_1,x_2, \cdots, x_{r-1}, b^*),a^*\rangle.\label{CYC_I}
\end{eqnarray}
Similarly, the evaluation
\begin{eqnarray}(x_1,x_2,\cdots, x_{r-1})
\big(\langle a, b^1\rangle ( b^2 \cdots b^r)\big)
&=&(a^*, x_1, x_2, \cdots,  x_{r-1})(b^1\cdots  b^r)\nonumber\\
&=&(a^*, x_{1}, x_2, \cdots,  x_{r-1})\triangle_r(b)\nonumber\\
&=&\mu_r (a^*, x_1,x_2,  \cdots, x_{r-1})(b)\nonumber\\
&=&\langle\mu_r(a^*, x_1, x_2,\cdots, x_{r-1}), b^*\rangle.\label{CYC_II}
\end{eqnarray}
The most right hand side of (\ref{CYC_I}) equals the most right hand side of (\ref{CYC_II}) with
sign counted
if and only if (\ref{Ainfty_pairing}) holds, and thus the lemma follows.
\end{proof}

\begin{remark}
By repeatedly applying (\ref{Ainfty_pairing}) to (\ref{CYC_I}) or to (\ref{CYC_II}) one gets more identities like
\begin{equation}\label{general_cyclic_pairing}
\langle a^m, b \rangle \cdot a^{m+1}\cdots a^ra^1\cdots a^{r-1}=
\pm\langle a,b^{\ell}\rangle\cdot b^{1}\cdots b^{\ell-1}b^{\ell+1}\cdots b^r.
\end{equation}
We leave the check to the interested reader.
\end{remark}

\subsection{Double bracket  on the cobar construction.}

The following result has already appeared in \cite[Theorem 15]{BCER}
when $C$ is a cyclic coalgebra (not $A_{\infty}$).

\begin{lemma}
\label{NCP_Coalg}
Let $C$ be a cyclic $A_\infty$ coalgebra of degree $-d$
and let $R=\mathbf{\Omega}_\infty( \tilde C)$.
Then we define
$\ldb-,-\rdb: \bar{R} \times \bar{R}
\to \bar{R} \otimes \bar{R}$
by
\begin{eqnarray}
\label{def_bb}
\ldb r, q \rdb &:=& \displaystyle
\sum_{\binom{i=1,\cdots,k}{j=1,\cdots,m}}(-1)^{|v_i|+\epsilon} \langle v_i, w_j\rangle\cdot
(s^{-1}w_1\, \cdots \, s^{-1}w_{j-1}\, s^{-1}v_{i+1} \, \cdots \, s^{-1}v_k)  \nonumber \\
&&\quad\quad\quad\quad  \otimes  (s^{-1} v_1\, \cdots\,  s^{-1}v_{i-1} \, s^{-1}w_{j+1} \, \cdots \, s^{-1}w_m)\, ,
\end{eqnarray}
for  $ r = (s^{-1} v_1\, s^{-1} v_2 \, \cdots\,  s^{-1} v_k)$ and $ q = (s^{-1}w_1\, s^{-1} w_2 \, \cdots \, s^{-1}w_m)$,
where $\epsilon$ is
$$(|r|+d) (|s^{-1}w_1|+\cdots+|s^{-1}w_{j-1}|)+(|s^{-1}v_1|+
\cdots+|s^{-1}v_{i-1}|+|s^{-1} w_j|+d)(|s^{-1}v_{i+1}|+\cdots+|s^{-1}v_k|).$$
Then the bracket \eqref{def_bb} gives a DG double $n$-Poisson structure on $R$,
where $n=2-d$ .
\end{lemma}
\begin{proof}
The proof is essentially the same as in \cite[Theorem 15]{BCER}.
We only need to show that the bracket commutes with the differential, which involves
the $A_\infty$ operators.

From the definition of the double bracket,
$$\ldb d(a_1a_2\cdots a_m), b_1b_2\cdots b_n\rdb+(-1)^{|a|}\ldb a_1a_2\cdots a_m, d(b_1b_2\cdots b_n)\rdb
-d\left(\ldb a_1a_2\cdots a_m, b_1b_2\cdots b_n\rdb\right)$$
contains summands whose coefficient
are pairings of components of $\triangle(a)$ with $b$, or pairings of $a$ with components of $\triangle(b)$, namely,
\begin{eqnarray*}
&&\sum_{i, j}\sum_{k}\langle a_i^k, b_j\rangle
b_1b_2\cdots b_{j-1}a_i^{k+1}\cdots a_i^ra_{i+1}\cdots a_m
\otimes a_1 a_2\cdots a_{i-1}a_i^1\cdots a_i^{k-1}b_{j+1}\cdots b_n\\
&-&
\sum_{i, j}\sum_{\ell}\langle a_i, b_j^{\ell}\rangle
b_1b_2\cdots b_{j-1}b_j^{1}\cdots b_j^{l-1}a_{i+1}\cdots a_m
\otimes a_1 a_2\cdots a_{i-1}b_j^{\ell+1}\cdots b_j^{s}b_{j+1}\cdots b_n,
\end{eqnarray*}
if we write $\triangle_{k}(a_i)=a_i^1a_i^2\cdots a_i^k$ and $\triangle_{\ell}(b_j)=b_j^1b_j^2\cdots b_j^\ell$.
However, from the cyclicity of the pairing (see (\ref{cyclic_pairing}) or its general form (\ref{general_cyclic_pairing})),
these two types of terms exactly cancel with each other.
Thus the double bracket commutes with the differential.
\end{proof}

\begin{remark}\label{bracket_ext}
One  can easily extend $\ldb-,-\rdb$ to $R$ by taking $\ldb r, 1\rdb =0$.
\end{remark}

\section{Brackets on the bimodule of one-forms}\label{Sect_bimod_of_1_form}

Let $C$ be a cyclic $A_\infty$ coalgebra and let $R:=\mathbf{\Omega}_\infty(\tilde C)$.
Using the double bracket on $R$ defined in Lemma~\ref{NCP_Coalg},
 we introduce a double bracket on $\Omega^1_{R}$ and the induced bracket on $\Omega^{1}_{R,\natural}$.
 Then we consider the corresponding Lie brackets on $\Omega^1_{R}$ and $\Omega^{1}_{R,\natural}$.

\subsection{Double bracket on $\Omega^1_{R}$.}

First we define a double bracket  
$\ldb-,-\rdb_{R \otimes R}: R \times (R\otimes R) \to R \otimes (R\otimes R)
\oplus (R\otimes R)\otimes R$
as follows
\begin{equation}
\label{doubletensor}
\ldb r, p\otimes q \rdb_{R \otimes R} \ := \ \ldb r, p\rdb \otimes q +(-1)^{|p|(|r|+n)} p\otimes \ldb r, q\rdb,
\end{equation}
where $\ldb - ,-\rdb$ is a double bracket on $R$ defined in Lemma~\ref{NCP_Coalg},
and in the right-hand side of \eqref{doubletensor},
the first summand lies in $R \otimes (R\otimes R)$ and the second
lies in $(R\otimes R)\otimes R$. Then one can show
\begin{lemma}
$(R \otimes R, \,  \ldb - , - \rdb_{R \otimes R})$ is a DG double $n$-Poisson $R$-bimodule.
\end{lemma}

\begin{proof}
 We need to verify axioms $(i)-(iii)$ for the double bimodule.
We first verify
$$ (i) \quad  \ldb r, q \cdot (p_1 \otimes p_2)\rdb_{R\otimes R} = \ldb r, q \rdb (p_1 \otimes p_2) +(-1)^{|q|(|r|+n)}
q\, \ldb r, p_1 \otimes p_2 \rdb_{R\otimes R} \, .$$
By definition the left hand side (LHS) is
\begin{eqnarray}
 \ldb r, \, q \cdot p_1 \otimes p_2 \rdb_{R\otimes R} &=& \ldb r, q \rdb  \cdot (p_1 \otimes p_2) +
 (-1)^{|q|(|r|+n)} q \cdot \ldb r, p_1 \rdb \otimes p_2  \nonumber \\[0.1cm]
 &+&   (-1)^{(|q|+|p_1|)(|r|+n)} q \cdot p_1 \otimes  \ldb r, p_2 \rdb    \nonumber
\end{eqnarray}
which is identically equal to the right hand side (RHS). Next we verify
$$ (ii) \quad  \ldb r \cdot q, p_1\otimes p_2 \rdb_{R\otimes R} = r \ast \ldb q, p_1\otimes p_2 \rdb_{R\otimes R}
+(-1)^{|q|(|p_1|+|p_2|+n)} \ldb r, p_1\otimes p_2 \rdb_{R\otimes R} \ast q\, . $$
Indeed, the LHS is equal to
\begin{eqnarray}
 && \ldb r \cdot q, p_1\rdb \otimes p_2 + (-1)^{|p_1|(|r|+|q|+n)}  p_1 \otimes \ldb r \cdot q, p_2 \rdb  \nonumber \\[0.2cm]
 &=&    r \ast \ldb q, p_1 \rdb \otimes p_2 + (-1)^{|q|(|p_1|+n)}  \ldb r, p_1 \rdb \ast q \otimes p_2  \nonumber  \\[0.2cm]
 &+&  (-1)^{|p_1|(|r|+|q|+n)}  p_1 \otimes \big( r \ast \ldb q, p_2 \rdb  + (-1)^{|q|(|p_2|+n)} \ldb r, p_2\rdb\ast q\big) \nonumber
\end{eqnarray}
and the RHS is
\begin{eqnarray}
&& r \ast \ldb q, p_1 \rdb \otimes p_2 + (-1)^{|p_1|(|r|+|q|+n)}  p_1 \otimes  r \ast \ldb q, p_2 \rdb \nonumber \\[0.2cm]
&+&  (-1)^{|q|(|p_1|+|p_2|+n) +|q||p_2|} \ldb r, p_1\rdb \ast q \otimes p_2  \nonumber\\[0.2cm]
&+& (-1)^{|q|(|p_1|+|p_2|+n)+ |p_1|(|r|+n)}
 p_1 \otimes \ldb r, p_2\rdb\ast q \nonumber
\end{eqnarray}
and hence they are equal. Similarly, we can prove the Jacobi identity $(iii)$.

Next we show that $\ldb-,-\rdb_{R \otimes R} $ commutes with the differential. Recall
that $d_{R\otimes R}= d_{R} \otimes 1 + 1 \otimes d_R$. So we need to prove
$$  d_R \, \ldb r, p \otimes q \rdb_{R \otimes R}  =  \ldb d_R(r), p \otimes q \rdb_{R \otimes R}
+(-1)^{|r|+n} \ldb r, d_R( p \otimes q ) \rdb_{R \otimes R} . $$
Indeed, using  \eqref{doubletensor}  one has
\begin{eqnarray}
d_R \, \ldb r, p \otimes q \rdb_{R \otimes R} &= & d_R \ldb r, p \rdb \otimes q + (-1)^{|r|+|p|+n}
\ldb r, p \rdb \otimes d_R(q)\nonumber \\[0.2cm]
&+ & (-1)^{|p|(|r|+n)} d_R(p) \otimes \ldb r, q\rdb + (-1)^{|p|(|r|+n+1)} p \otimes d_R  \ldb r, q\rdb   \, ;  \nonumber \\[0.4cm]
\ldb d_R(r), p \otimes q \rdb_{R \otimes R} &=&  \ldb d_R(r), p  \rdb \otimes q + (-1)^{|p|(|r|+n+1)}
 p\otimes \ldb d_R(r), q \rdb  \, ;  \nonumber\\[0.4cm]
 \ldb r, d_R( p \otimes q ) \rdb_{R \otimes R} &=& \ldb r, d_R(p) \otimes q \rdb +(-1)^{|p|} \ldb r,  p\otimes d_R(q) \rdb \nonumber\\[0.2cm]
 &=& \ldb r, d_R(p) \rdb \otimes q +(-1)^{(|p|+1)(|r|+n)} d_R(p) \otimes \ldb r, q \rdb \nonumber\\[0.2cm]
&+&  (-1)^{|p|} \ldb r,  p \rdb \otimes d_R(q)  + (-1)^{|p|(|r|+n+1)}  p  \otimes  \ldb r, d_R(q) \rdb . \nonumber
\end{eqnarray}
Combining the RHS of these identities we get
\begin{eqnarray}
&&  \big( d_R \ldb r, p \rdb - \ldb d_R(r), p  \rdb - (-1)^{|r|+n} \ldb r, d_R(p) \rdb \big) \otimes q \nonumber \\[0.2cm]
&& + (-1)^{|p|(|r|+n+1)} p \otimes \big( d_R  \ldb r, q\rdb - \ldb d_R(r), q \rdb - (-1)^{|r|+n} \ldb r, d_R(q) \rdb \big)
\end{eqnarray}
and this expression is equal to $0$, since by Lemma~\ref{NCP_Coalg} $\ldb -,-\rdb$ commutes with $d_R$.
This finishes our proof.
\end{proof}

We claim that the bracket in \eqref{doubletensor} can be restricted to $\Omega^{1}_{R}$.
Recall that $\Omega^{1}_{R} \cong R \, \otimes \, s^{-1}C \,  \otimes\, R$
and $\Omega^{1}_{R,\natural} \cong s^{-1}C  \otimes R$, where identifications are
given by the map $\mathrm{I}$ defined in \eqref{mapI}.
Indeed, let
$$ \omega= (s^{-1}v_1\, ... \,s^{-1}v_{p-1}) \otimes s^{-1}v_p \otimes (s^{-1}v_{p+1}\,  ...\, s^{-1}v_m)
\in \Omega^{1}_{R} \,.$$
Then ${\rm I}(\omega)=b \cdot s^{-1}v_p \otimes c - b\otimes s^{-1}v_p \cdot c  $, where
$b=(s^{-1}v_1\, ...\, s^{-1}v_{p-1})$ and $c=(s^{-1}v_{p+1} \, ... \, s^{-1}v_m)$ and one has
\begin{eqnarray}
 \ldb a, {\rm I}(\omega)  \rdb_{R \otimes R}  &= & \ldb a, b \cdot s^{-1}v_p \rdb \otimes c  +
(-1)^{(|b|+|s^{-1}v_p)(|a|+n)} b \cdot s^{-1}v_p \otimes  \ldb a, c \rdb   \nonumber \\[0.2cm]
&  -& \ldb a, b\rdb \otimes  s^{-1}v_p \cdot c -(-1)^{|b|(|a|+n)} b \otimes  \ldb a, s^{-1}v_p \cdot c \rdb \nonumber \\[0.2cm]
\label{rest1}
& = &\ldb a, b \rdb^{(1)} \otimes  \ldb a, b \rdb^{(2)} \cdot s^{-1}v_p \otimes c  \\[0.2cm]
\label{rest2}
& + &(-1)^{|b|(|a|+n)} b\cdot \ldb a, s^{-1}v_p \rdb^{(1)}
\otimes  \ldb a, s^{-1}v_p \rdb^{(2)} \otimes c \\[0.2cm]
\label{rest3}
& + &(-1)^{(|b|+|s^{-1}v_p|)(|a|+n)} b \cdot s^{-1}v_p \otimes  \ldb a, c \rdb^{(1)} \otimes  \ldb a, c \rdb^{(2)} \\[0.2cm]
\label{rest4}
&  -  &\ldb a, b\rdb^{(1)} \otimes \ldb a, b\rdb^{(2)} \otimes s^{-1}v_p \cdot c  \\[0.2cm]
\label{rest5}
& - &(-1)^{|b|(|a|+n)}\, b \otimes  \ldb a, s^{-1}v_p \rdb^{(1)} \otimes  \ldb a, s^{-1}v_p \rdb^{(2)} \cdot c \\[0.2cm]
\label{rest6}
& - & (-1)^{(|b|+|s^{-1}v_p|)(|a|+n)} \, b \otimes s^{-1}v_p \cdot  \ldb a, c \rdb^{(1)} \otimes  \ldb a, c \rdb^{(2)}.
\end{eqnarray}
Now \eqref{rest1}+\eqref{rest4} $ \in R \otimes  \Omega^{1}_{R}$ and \eqref{rest3}+ \eqref{rest6}
$ \in \Omega^{1}_{R} \otimes R$. On the other hand
\begin{eqnarray}
\eqref{rest2} - (-1)^{|b|(|a|+n)} b \otimes \ldb a, s^{-1}v_p \rdb^{(1)} \cdot  \ldb a, s^{-1}v_p \rdb^{(2)} \otimes c
  &\in & \Omega^{1}_{R} \otimes R , \label{doublemodulestr_1}\\[0.1cm]
 \eqref{rest5} +   (-1)^{|b|(|a|+n)} b \otimes \ldb a, s^{-1}v_p \rdb^{(1)} \cdot  \ldb a, s^{-1}v_p \rdb^{(2)} \otimes c
 &\in & R \otimes  \Omega^{1}_{R}.\label{doublemodulestr_2}
 \end{eqnarray}
These two formulas hold since, 
for example in \eqref{doublemodulestr_1},
by taking out the third common component $c$ in the tensor,
the first two components exactly lie in $\Omega_R^1$.
Thus,  we define $\ldb - , - \rdb_{\Omega^{1}_{R}} : R \times \Omega^{1}_{R} \to
(R \otimes  \Omega^{1}_{R}) \oplus ( \Omega^{1}_{R} \otimes R) $ as
\begin{equation}
\label{defOmegabracket}
\ldb a, \omega \rdb_{\Omega^{1}_{R}}\, := \, \ldb a, {\rm I}(\omega) \rdb_{R \otimes R}
\end{equation}
and we have proved
\begin{corollary}
\label{cor22}
$(\Omega^{1}_{R}, \ldb - , - \rdb_{\Omega^{1}_{R}})$ is a DG double $n$-Poisson sub-bimodule
of $R\otimes R$.
\end{corollary}

\subsection{Lie brackets on $\Omega^1_{R}$ and $\Omega^1_{R,\natural}$}

Let us recall some notations from the previous section. First, $\{-,-\}:= \mu_{R} \circ \ldb-,-\rdb$,
where $\mu_{R}$ is the multiplication map on $R$. Second,
$$ \{-,-\}_{\Omega^1_{R}}:=
\mu_{ \Omega^1_{R}} \circ \ldb - , - \rdb_{\Omega^{1}_{R}} \, , \quad
\{-,-\}_{\Omega^1_{R, \natural}} := \natural \circ  \{-,-\}_{\Omega^1_{R}} \, , $$
where $\mu_{ \Omega^1_{R}}$
is the bimodule action map.
\begin{theorem}
\label{bracksmain2}
Let $C$ be a cyclic $A_\infty$ coalgebra of degree $-d$
and let $R=\overline{\mathbf{\Omega}}_\infty( \tilde C)$.
Let $n:= 2-d$.
Then
\begin{enumerate}
\item[($a$)]
$\{-,-\}$ induces a DG $n$-Lie algebra structure on  $R_{\natural}$
and a DG $n$-Lie module on $R$.
\item[($b$)] $\{-,-\}_{\Omega^1_{R}}$ and $\{-,-\}_{\Omega^1_{R, \natural}}$
make $\Omega^1_{R}$ and $\Omega^1_{R, \natural}$ into
DG $n$-Lie modules over   $R_{\natural}$.
\end{enumerate}
In particular, $\mathrm{H}_{\bullet}(R_{\natural})$ is a graded $n$-Lie algebra
and  $\mathrm{H}_{\bullet}(R)$, $\mathrm{H}_{\bullet}( \Omega^1_{R})$
and $\mathrm{H}_{\bullet}( \Omega^1_{R, \natural})$ are $n$-Lie modules.
\end{theorem}

\begin{proof}
Follows from Proposition~\ref{mainprop} for the double bracket on $R$ defined
in Lemma~\ref{NCP_Coalg} and the double bracket on $\Omega^1_{R}$
defined in Corollary~\ref{cor22}.
\end{proof}

Now we turn our attention to the $2$-periodic sequence of complexes

$$\xymatrix{  \ar[r]^{\bar{\partial}}& \Omega^1_{R,\natural}  \ar[r]^{-\beta}& \overline{R} \ar[r]^{\bar{\partial}}&
 \Omega^1_{R,\natural}  \ar[r]^{-\beta}& } $$
which  by Proposition~\ref{propbicomp} is identical to the cyclic bicomplex of $C$.  Our claim is

\begin{theorem}
\label{Liemorph1}
$\bar{\partial}$  and  $\beta$ are morphisms of Lie modules, that is,
\begin{equation}
\label{Liemorph}
 \mathrm{I} ( \bar{\partial} \{ r, q \} ) = \{r, \bar{\partial}(q)\}_{\Omega^1_{R,\natural}}
\, , \quad \beta \{r,\omega\}_{\natural}= \{r,\beta(\omega)\}.
\end{equation}
\end{theorem}

We need the following auxiliary statement

\begin{lemma}
\label{bracketidentononeform}
Let  $r \in R$ and $ \omega=q\otimes s^{-1}v \in \Omega^1_{R,\natural}$. Then
\begin{eqnarray}
 \{r, \omega\}_{\Omega^1_{R,\natural}} &=& \natural \circ  \bigg[ \{r,q\}s^{-1}v \otimes 1 - \{ r,q \} \otimes s^{-1}v
 \nonumber\\
 \label{lastexp}
&+ & (-1)^{|q|(|r|+n)} \big( q \cdot \{r,s^{-1}v\} \otimes 1 -  q \otimes \{r,s^{-1}v\} \big) \bigg].
\end{eqnarray}
\end{lemma}

\begin{proof}
It follows directly from the definition of the double bracket on $\Omega^1_R$. Indeed, we have
\begin{eqnarray*}
&& \{r, \omega\}_{\Omega^1_R}\\
&=  &\mu \circ \ldb r,\ {\rm I}(\omega)\rdb_{R \otimes R}    \nonumber \\[0.2cm]
&= &\mu_{\Omega^1_{R}}  \circ \big[  \ldb r, q\cdot s^{-1}v \otimes 1  \rdb_{R \otimes R}  -
\ldb r, q \otimes s^{-1}v \rdb_{R \otimes R} \big] \nonumber \\[0.2cm]
& = &\mu_{\Omega^1_{R}} \circ \big [ \ldb r, q \rdb \cdot s^{-1}v  \otimes 1
- \ldb r,q \rdb \otimes s^{-1}v \big] \nonumber\\[0.2cm]
& + & (-1)^{|q|(|r|+n)} \mu_{\Omega^1_{R}} \circ \big [ q \cdot \ldb r, s^{-1}v\rdb \otimes 1
 - q \otimes \ldb r, s^{-1}v\rdb \big] \nonumber\\[0.2cm]
&=& \{r,q\}s^{-1}v \otimes 1 - \{ r,q \} \otimes s^{-1}v \nonumber \\[0.2cm]
&+&(-1)^{|q|(|r|+n)} \mu_{\Omega^1_{R}} \circ
\big[ q \cdot \ldb r, s^{-1}v\rdb^{(1)} \otimes   \ldb r, s^{-1}v\rdb^{(2)} \otimes 1
 - q \otimes \{r,s^{-1}v\} \otimes 1\big  ]  \nonumber\\[0.2cm]
 & +&(-1)^{|q|(|r|+n)} \mu_{\Omega^1_{R}} \circ \big[  q \otimes \{r,s^{-1}v\} \otimes 1 -
 q\otimes  \ldb r, s^{-1}v\rdb^{(1)} \otimes   \ldb r, s^{-1}v\rdb^{(2)} \big] \nonumber\\[0.2cm]
 & = & \{r,q\}s^{-1}v \otimes 1 - \{ r,q \} \otimes s^{-1}v \nonumber    \\[0.2cm]
 &+&(-1)^{|q|(|r|+n)} \big( q \cdot \ldb r,s^{-1}v\rdb^{(1)}  \otimes
 \ldb r, s^{-1}v\rdb^{(2)}  - q \otimes \{r,s^{-1}v\} \big)\nonumber\\[0.2cm]
 &+ &(-1)^{|q|(|r|+n)} \big( q \cdot \{r,s^{-1}v\} \otimes 1 -  q \cdot \ldb r,s^{-1}v\rdb^{(1)}
 \otimes \ldb r, s^{-1}v\rdb^{(2)}  \big) \nonumber \\[0.2cm]
 &= &\{r,q\} \cdot s^{-1}v \otimes 1 - \{ r,q \} \otimes s^{-1}v
 + (-1)^{|q|(|r|+n)} \big( q \cdot \{r,s^{-1}v\} \otimes 1 -  q \otimes \{r,s^{-1}v\} \big). \nonumber\qedhere
\end{eqnarray*}
\end{proof}

\begin{proof}[Proof of Theorem~\ref{Liemorph1}]
To prove the first identity of \eqref{Liemorph}, it suffices to show
$$ \mathrm{I}( \partial \{r, q\}) = \{r, \partial (q)\}_{\Omega^1_R} \, .$$
By \eqref{Ipartial}, we have  $ \mathrm{I}( \partial \{r, q\})= \{r, q\} \otimes 1 -1 \otimes \{r, q\}$.
On the other hand
\begin{eqnarray}
\{r, \partial (q)\}_{\Omega^1_R} &=& \mu_{\Omega^1_R} \ldb r, \mathrm{I}( \partial (q)) \rdb_{R \otimes R}
=  \mu_{\Omega^1_R} \ldb r,  q \otimes 1 - 1 \otimes q \rdb_{R\otimes R}  \nonumber \\[0.1cm]
&= & \mu_{\Omega^1_R} \big[ \ldb r,  q \rdb  \otimes 1 - 1 \otimes \ldb r, q \rdb \big]
= \{r, q\} \otimes 1 -1 \otimes \{r, q\} .\nonumber
\end{eqnarray}

Let $\omega=q\otimes s^{-1}v$ for some $q \in R$ and $v \in C$. Then
$$\beta(\omega)=q\cdot s^{-1}v - (-1)^{|s^{-1}v||q|} s^{-1}v \cdot q\, , \quad
{\rm I} (\omega) = q\cdot s^{-1}v \otimes 1 - q \otimes s^{-1}v  \, , $$
 and
\begin{eqnarray}
 \{r,\beta(\omega)\} &=& \{ r,  q\cdot s^{-1}v\} -  (-1)^{|s^{-1}v||q|} \{r, s^{-1}v \cdot q\}  \nonumber\\[0.1cm]
& =& \{r,q\} \cdot s^{-1}v  + (-1)^{|q| (|r|+n)} q\cdot \{r,s^{-1}v\} \nonumber \\[0.1cm]
\label{bracbetaomega}
& -& (-1)^{|s^{-1}v||q|} \{r, s^{-1}v\} \cdot q - (-1)^{|s^{-1}v|(|q|+|r|+n)} s^{-1}v  \cdot \{r,q\} \, .
 \end{eqnarray}
On the other hand, by the commutative diagram \eqref{comdiag} for $R$, we  have
\begin{equation}
\beta( \{r, \omega\}_{\natural})= \mu\circ \sigma ({\rm I}\, \{r, \omega\}_{\Omega^1_R}) \, .
\end{equation}
Using \eqref{lastexp} , one has
\begin{eqnarray}
\beta( \{r, \omega\}_{\natural}) &= &\{r,q\} \cdot s^{-1}v - (-1)^{|s^{-1}v|(|r|+|q|+n)} s^{-1}v \cdot \{ r,q \} \nonumber\\[0.2cm]
\label{betaromega2}
& + &(-1)^{|q|(|r|+n)} q \cdot \{r,s^{-1}v\}  - (-1)^{|s^{-1}v||q|} \{r,s^{-1}v\} \cdot q.
\end{eqnarray}
We finish our proof by comparing  \eqref{bracbetaomega}  and \eqref{betaromega2}.
\end{proof}

\section{$N$-Koszul Calabi-Yau algebras}\label{Sect_N_Kosuzl_CY}

\begin{definition}[Ginzburg \cite{Ginzburg}]
An associative algebra $A$ is said to be a \textit{Calabi-Yau algebra} of dimension $d$
if $A$ is homologically smooth and there exists an isomorphism
$$
\eta:\mathrm{RHom}_{A\otimes A^{\mathrm{op}}} (A, A\otimes A)\rightarrow A[-d]
$$
in the derived category of $A\otimes A^{\mathrm{op}}$-modules.
\end{definition}

In the above definition,
an algebra $A$ is said to be {\it homologically smooth}
if $A$ is a perfect $A\otimes A^{\mathrm{op}}$-module,
i.e. it has a finitely-generated projective resolution of finite length.

Ever since they are first introduced by Ginzburg, Calabi-Yau algebras have been widely
studied by mathematicians from various fields.

So far, most of Calabi-Yau algebras appeared in literature are Koszul or $N$-Koszul,
(the most general case is due to Van den Bergh \cite{VdB3}),
and hence they are all of the form
$$\mathbf{\Omega}_\infty(A^{\ac}) \stackrel{\simeq}{\twoheadrightarrow}A,$$
where $A^{\ac}$ is the Koszul dual $A_\infty$ coalgebra of $A$,
which is a cyclic $A_\infty$ coalgebra.

We now recall the definition of $N$-Koszul algebras.
Let $V$ be a finite dimensional vector space over $k$,
and $S$ be a subspace of $V^{\otimes N}$,
where $N\geq 2$ is an integer.
Let $TV$ be the tensor algebra of $V$, and $\langle S\rangle$ be the two-sided ideal of $TV$ generated by $S$.
The quotient algebra $A:=TV/\langle S\rangle$ is called an \emph{$N$-homogeneous algebra},
and is denoted by $A=A(V,S)$.
The $N$-homogenous algebra $A^{\vee}:=TV^{\ast}/\langle S^{\bot} \rangle$ is
called \emph{$N$-homogeneous dual algebra},
where $V^{\ast}$ is the dual space of $V$ and $S^{\bot}\subset (V^{\ast})^{\otimes N}$
is the orthogonal complement of $S$ in $(V^{\ast})^{\otimes N}$.
The \textit{$N$-Koszul dual algebra} $A^{!}$ of the $N$-homogenous algebra $A$ is defined as follows: set
$$
N(i)=
\left\{
\begin{array}{ll}Nj,&\mbox{if}\ i=2j,\\
Nj+1,&\mbox{if}\ i=2j+1,
\end{array}
\right.
$$
and define $A^{!}:=\bigoplus_{i} A^{\vee}_{N(i)}$. We also denote
$A^{\ac}:=(A^{!})^{\ast}\subset TV$, which is called \emph{$N$-Koszul dual coalgebra} of $A$.
Here we use the convention that the dual of a graded space
with finite-dimensional component is the direct sum of component-wise duals.

\begin{definition}[Berger \cite{Berger}, $N$-Koszul algebra]
An $N$-homogeneous algebra $A$ is (left) {\it $N$-Koszul} if the trivial left $A$-module
$_Ak$ admits a linear projective resolution
$$
\xymatrix@C=0.5cm{
 \cdots \ar[r] & P_i \ar[r]^{b} & P_{i-1} \ar[r]^{b} & \cdots \ar[r]^{} & P_1 \ar[r]^{b} & P_0 \ar[r]^{b} & _Ak }
$$
with $P_i=A\otimes_k A^{\ac}_i$, and if we choose a basis $\{e_i\}$ for $V$
and let $\{e^*\}$ be the dual basis, then the differential $b=\sum_i^{\dim V}e_i\otimes e_i^*$.
\end{definition}

One can define the right $N$-Koszul algebra similarly. In \cite{Berger}, it is proved that an $N$-homogeneous
algebra is left $N$-Koszul if and only if it is right $N$-Koszul.
Note that $2$-Koszul algebras are Koszul algebras in the usual sense.
For a $2$-Koszul algebra $A$, it is known that the Yoneda algebra $\mathrm{Ext}^\bullet_A(k, k)$
is an associative algebra. The $A_\infty$ algebra structure on $\mathrm{Ext}^\bullet_A(k,k)$ is started by
Lu, Palmieri, Wu and Zhang in papers \cite{LPWZ1,LPWZ2}.

\begin{proposition}[Berger-Marconnet; He-Lu]
Let $A$ be an $N$-Koszul algebra. Then its Yoneda algebra $\mathrm{Ext}^\bullet_A(k, k)$
is isomorphic to its $N$-Koszul dual algebra $A^{!}$ as unital, augmented $A_\infty$ algebras.
\end{proposition}

\begin{proof}
See Berger-Marconnet \cite[Proposition 3.1]{BM}, or He-Lu \cite[Theorem 6.5]{HL}.
\end{proof}

More precisely, the authors cited above showed that for $N$-Koszul algebras,
all the $A_\infty$ operators but $m_2, m_N$ vanish.
Now if $A$ is an $N$-Koszul algebra,
then $A^{\ac}$ is a co-unital, co-augmented $A_\infty$ coalgebra.

\begin{theorem}[Dotsenko-Vallette]\label{LVDV}
Let $A$ be an $N$-Koszul algebra, and $A^{\ac}$ be its $N$-Koszul dual $A_\infty$
coalgebra. Then
the morphism of DG algebras $\mathbf{\Omega}_\infty( A^{\ac})\stackrel{\simeq}{\twoheadrightarrow} A$ is a quasi-isomorphism.
\end{theorem}

\begin{proof}The $N=2$ case is standard, see, for example, Loday-Vallette \cite[Theorem 3.4.6]{LV}.
For the $N$-Koszul case, see Dotsenko-Vallette \cite[Theorem 5.1]{DV}.
\end{proof}

\begin{proposition}[He-Van Oystaeyen-Zhang; Wu-Zhu]\label{Thm_HL}
Let $A$ be an $N$-Koszul algebra ($N\geq 3$). Then the following conditions are equivalent:
\begin{enumerate}
  \item $A$ is a $d$-Calabi-Yau algebra;
  \item $\mathrm{Ext}^\bullet_A(k,k)$ is a unital and augmented cyclic $A_{\infty}$ algebra of degree $-d$.
\end{enumerate}
\end{proposition}

\begin{proof}
See Wu-Zhu \cite[Corollary 4.11]{WZ1} and He-Van Oystaeyen-Zhang \cite[Proposition 3.3]{HVZ11}.\end{proof}

In view of Lemma~\ref{equivofalgcoal},
we can summarize the above three results in terms of {\it $A_\infty$ coalgebras} in the following form (due to the works
of Berger, Dotsenko-Vallette, He-Lu, He-Van Oystaeyen-Zhang, Van den Bergh and Wu-Zhu cited above):

\begin{theorem}\label{Thm_VdB}
Let $A$ be an $N$-Koszul algebra.
Then $A$ is $d$-Calabi-Yau if and only if
the Koszul dual $A_\infty$ coalgebra $A^{\ac}$ is cyclic of degree $-d$.
Moreover, in this case $\mathbf{\Omega}_\infty(A^{\ac})$ is a cofibrant resolution of $A$.
\end{theorem}

The most general form of this theorem is due to Van den Bergh given in \cite[cf. Theorem 11.1]{VdB3}.
Van den Bergh does not use the terminology ``$N$-Koszul'',
but it is clear that all $N$-Koszul algebras are Koszul in the sense of \cite{VdB3},
where the latter is the linear dual of the bar construction.
Since the finite dimensionality $\mathrm{Ext}^\bullet_{A}(k, k)$ is essentially used in the following construction,
we focus only on the $N$-Koszul case.

\subsection{Examples}
In this subsection, we list several known examples of $N$-Koszul Calabi-Yau
algebras, and therefore they all admit a derived non-commutative Poisson structure.
Note that by the above theorem, we only need to describe the cyclic $A_\infty$ algebra
structure on their Koszul dual.

First let us remind that
all {\it graded} $2$- and $3$-Calabi-Yau algebras are $N$-Koszul.
This fact is proved by Berger-Marconnet in \cite[Proposition 5.2]{BM}.
Being graded is important here,
since from Davison's result there exist $3$-Calabi-Yau algebras which
are not superpotential algebras and hence not $N$-Koszul (\cite{Davison}). %

\subsubsection{Three dimensional Sklyanin algebras}
Let $a,b,c\in k$.
The three dimensional Sklyanin algebra $A=A(a,b,c)$
is the graded $k$-algebra with generators $x, y, z$ of degree one, and relations
$$
\begin{array}{c}
f_1=cx^2+bzy+ayz=0,\\
f_2=azx+cy^2+bxz=0,\\
f_3=byx+axy+cz^2=0.
\end{array}
$$
This Sklyanin algebra is one of the important examples of Ginzburg's
Calabi-Yau algebras (see \cite[Example 1.3.8]{Ginzburg}).

Smith showed in \cite[Example 10.1]{Smith}
that $A$ is Koszul, whose dual algebra $A^!$
is generated by $\xi_1,\xi_2,\xi_3$ with relations
\begin{eqnarray*}
c\xi_2\xi_3-b\xi_3\xi_2,&& b\xi_1^2-a\xi_2\xi_3,\\
c\xi_3\xi_1-b\xi_1\xi_3,&& b\xi_2^2-a\xi_3\xi_1,\\
c\xi_1\xi_2-b\xi_2\xi_1,&& b\xi_3^2-a\xi_1\xi_2.
\end{eqnarray*}
To describe the non-degenerate pairing, choose a basis for $A^!$:
$$\begin{array}{cl}
A_0^!: & 1\\
A_1^!: & \xi_1,\xi_2,\xi_3\\
A_2^!: &\xi_1^2,\xi_2^2,\xi_3^2\\
A_3^!:&\xi_1\xi_2\xi_3
\end{array}
$$
The pairing $\langle u,v\rangle$ for homogeneous $u,v\in A^!$
is given by the scalar of $uv$ with respect to $\xi_1\xi_2\xi_3$.
This pairing is cyclically invariant, and therefore
we obtain a derived non-commutative Poisson structure on $A$.

\subsubsection{Four dimensional Sklyanin algebras}

Let $\alpha, \beta,\gamma\in k$ such that
$$\alpha+\beta+\gamma+\alpha\beta\gamma=0,\quad \{\alpha,\beta,\gamma\}\cap\{0,\pm 1\}=\emptyset.$$
The four dimensional Sklyanin algebra $A=A(\alpha,\beta,\gamma)$
is the graded $k$-algebra with generators $x_0,x_1,x_2,x_3$ of
degree one, and relations $f_i=0$, where
\begin{eqnarray*}
f_1=x_0x_1-x_1x_0-\alpha(x_2x_3+x_3x_2),&& f_2=x_0x_1+x_1x_0-(x_2x_3-x_3x_2),\\
f_3=x_0x_2-x_2x_0-\beta(x_3x_1+x_1x_3),&& f_4=x_0x_2+x_2x_0-(x_3x_1-x_1x_3),\\
f_5=x_0x_3-x_3x_0-\gamma(x_1x_2+x_2x_1),&& f_6=x_0x_3+x_3x_0-(x_1x_2-x_2x_1).
\end{eqnarray*}
As proved by Smith and Stafford (\cite[Propositions 4.3-4.9]{SS}),
$A$ is Koszul, whose Koszul dual algebra $A^!$
is generated by $\xi_0,\xi_1,\xi_2,\xi_3$ with the following relations:
\begin{eqnarray*}
&&  \xi_0^2=\xi_1^2=\xi_2^2=\xi_3^2=0,\\
&&  2\xi_2\xi_3+(\alpha+1)\xi_0\xi_1-(\alpha-1)\xi_1\xi_0=0,\\
&&  2\xi_3\xi_2+(\alpha-1)\xi_0\xi_1-(\alpha+1)\xi_1\xi_0=0,\\
&&  2\xi_3\xi_1+(\beta+1)\xi_0\xi_2-(\beta-1)\xi_2\xi_0=0,\\
&&  2\xi_1\xi_3+(\beta-1)\xi_0\xi_2-(\beta+1)\xi_2\xi_0=0,\\
&&  2\xi_1\xi_2+(\gamma+1)\xi_0\xi_3-(\gamma-1)\xi_3\xi_0=0,\\
&&  2\xi_2\xi_1+(\gamma-1)\xi_0\xi_3-(\gamma+1)\xi_3\xi_0=0.
\end{eqnarray*}
Smith and Stafford also showed that $A^!$ admits a non-degenerate symmetric pairing.
To see this, $A^!$ is spanned by
the following elements:
$$\begin{array}{cl}
A_0^!: & 1\\
A_1^!: &\xi_0,\xi_1,\xi_2,\xi_3\\
A_2^!: &\xi_0\xi_1,\xi_0\xi_2,\xi_0\xi_3,\xi_1\xi_0,\xi_2\xi_0,\xi_3\xi_0\\
A_3^!:&\xi_0\xi_1\xi_0,\xi_0\xi_2\xi_0,\xi_0\xi_3\xi_0,\xi_1\xi_0\xi_1\\
A_4^!:&\xi_0\xi_1\xi_0\xi_1
\end{array}
$$
and all other degree components are zero. Also, in degree 4, we have the following identities
$$\begin{array}{ll}
\xi_0\xi_j\xi_0\xi_j=-\xi_j\xi_0\xi_j\xi_0\quad\mbox{for}\quad 1\le j\le 3,&\xi_0\xi_i\xi_0\xi_j=0\quad\mbox{for}\quad i\ne j,
\displaystyle\frac{}{}\\
\xi_0\xi_3\xi_0\xi_3=\displaystyle\frac{1+\alpha}{1-\gamma}\xi_0\xi_1\xi_0\xi_1,&\\
\xi_0\xi_2\xi_0\xi_2=\displaystyle\frac{1+\gamma}{1-\beta}\frac{1+\alpha}{1-\gamma}\xi_0\xi_1\xi_0\xi_1.&
\end{array}$$
The pairing $\langle a,b\rangle$ for homogeneous $a, b\in A^!$ is defined to be the scalar
of $ab$ with respect to $\xi_0\xi_1\xi_0\xi_1$.
One easily sees the pairing such defined is graded symmetric and is cyclic.
Therefore $A(\alpha,\beta,\gamma)$ is Koszul Calabi-Yau of dimension 4, and there is
derived non-commutative Poisson structure on it.

\subsubsection{Universal enveloping algebras}
Let us first say a bit about {\it linear-quadratic} Koszul algebras.
Suppose $V$ is a finite dimensional vector space. A {\it linear quadratic relation} is a subspace
$S\subset V\oplus V^{\otimes 2}$.
And we may define the linear quadratic algebra $A(V, S)$ as before.

In this subsection, we assume $S$ satisfies the following two conditions:
\begin{eqnarray}
S\cap V&=&0,\label{ql1}\\
\{S\otimes V+V\otimes S\}\cap V^{\otimes 2}&=&S\cap V^{\otimes 2}.\label{ql2}
\end{eqnarray}
Let $qS: S\to V^{\otimes 2}$ be the projection.

\begin{definition}[Linear quadratic Koszul algebra]
A linear quadratic algebra $A=A(V, S)$ is said to
be {\it Koszul} if it satisfies conditions
(\ref{ql1}) and (\ref{ql2}) and if the  quadratic
algebra $A(V, qS)$ is Koszul.
\end{definition}

Since $qS$ is the image of the projection of $S$ to $V^{\otimes 2}$,
we in fact have a map
$$
\phi: qS\to V
$$
such that $S=\{X-\phi(X)|X\in qS\}$.
Denote by $(qA)^{\ac}$ the quadratic dual coalgebra of $T(V)/(qS)$,
then this $\phi$ gives a map
$$
d_\phi: (qA)^{\ac}\twoheadrightarrow qS\to V,
$$
which extends to a coderivation $d_\phi: (qA)^{\ac}\to T(V)$.

Now if
\begin{equation*}
\{S\otimes V+V\otimes S\}\cap S^{\otimes 2}\subset qS,
\end{equation*}
then the images of $d_\phi$ lie in $(qA)^{\ac}$. We in fact get
a co-derivation
$$
d_\phi: (qA)^{\ac}\to (qA)^{\ac}.
$$
And if furthermore,
\begin{equation}\label{linear-quadratic}
\{S\otimes V+V\otimes S\}\cap S^{\otimes 2}\subset S\otimes V^{\otimes 2},
\end{equation}
then $(d_\phi)^2=0$,
and we obtain a differential graded coalgebra $((qA)^{\ac}, d_\phi)$,
which is the {\it Koszul dual coalgebra} of $A(V,S)$,
and is denoted by $(A(V,S)^{\ac},d)$. Its linear dual is a differential graded
algebra, and is denoted by $A(V,S)^!$.
For more details, see \cite[\S3.6]{LV}.

\begin{lemma}\label{KtoLK}
Suppose $A(V,qS)$ is a Calabi-Yau algebra of dimension $n$.
Then $A(V, S)$ is a Calabi-Yau algebra of the same dimension
if and only if any nonzero element in $A(V,S)^{\ac}_n$ is a cycle.
\end{lemma}

\begin{proof}
As underlying spaces, $A(V,S)^{!}$ is isomorphic to $A(V, qS)^{!}$.
Since $A(V,qS)$ is Calabi-Yau, there is a cyclically invariant non-degenerate
pairing on $A(V,qS)^{!}$, which then gives the same pairing on
$A(V,S)^{!}$.
Thus to show $A(V,S)$
is Calabi-Yau, it is equivalent to show such pairing respects
the differential, which is again equivalent to show the
fundamental chain (the top chain in $A(V,S)^{\ac}$)
is a cycle. This completes the proof.
\end{proof}

Suppose $\mathfrak g$ is a Lie algebra, then
the universal enveloping algebra $U(\mathfrak g)$
has Koszul dual differential graded coalgebra which is exactly the Chevalley-Eilenberg complex
$(\mathrm{CE}_\bullet(\mathfrak g), d)$.
It is also known ({\it cf}. \cite[Chapter V]{GHV}) that the top chain of the Chevalley-Eilenberg complex
is a cycle if and only if $\mathrm{Tr}(\mathrm{ad}(g))=0$ for all $g\in\mathfrak g$. A Lie algebra satisfying
this property is called {\it unimodular}.
Examples of unimodular Lie algebras are abelian Lie algebras, semi-simple Lie algebras, Heisenberg Lie algebras, and
the Lie algebra of compact groups.
Thus as a corollary to Lemma \ref{KtoLK},
we have the following statement, which is due to He, Van Oystaeyen and Zhang
(\cite[Theorem 5.3]{HVZ}).

\begin{theorem}[He, Van Oystaeyen and Zhang]
Let $\mathfrak g$ be a Lie algebra of dimension $n$.
Then the universal enveloping algebra $U(\mathfrak g)$ is $n$-Calabi-Yau
if and only if $\mathfrak g$ is unimodular.
\end{theorem}

The derived non-commutative Poisson structure on $U(\mathfrak g)$ is highly nontrivial, even
in the polynomial case ({\it i.e.} the case where $\mathfrak g$ is abelian);
for more details see \S\ref{Sect_example_poly_alg}.

\subsubsection{Yang-Mills algebras}
Yang-Mills algebras are introduced by Connes and Dubois-Violette (\cite{CDV}).
An algebra $A$ is called a \emph{Yang-Mills algebra} if $A$ is generated by
the elements $x_i$ ($i\in\{ 1,\cdots,n\}$) with the following relations:
$$
g^{ij}[x_i,[x_j,x_l] ] =0, \, l\in \{ 1,2,\cdots,n\},
$$
where $(g^{ij})$ is a symmetric invertible $n\times n$-matrix.
Equivalently, $A=A(V,S)$ is a $3$-homogenous algebra,
with $V:=\bigoplus_{i}k x_i$, and $S\subset V^{\otimes 3}$ spanned by
$$
g^{ij}(x_i\otimes x_j\otimes x_l+x_l\otimes x_i\otimes x_i-2 x_i\otimes x_l\otimes x_j).
$$
In literature, the above $A$ is also denoted by $\mathrm{YM}(n)$.

Yang-Mills algebras are $3$-Koszul $3$-Calabi-Yau.
In fact Connes and Dubois-Violette (\cite[Theorem 1]{CDV}) proved that Yang-Mills algebras are $3$-Koszul, and
Berger-Taillefer (\cite[Proposition 4.5]{BT}) proved they are $3$-Calabi-Yau.

We give a brief description of the $A_\infty$ operators and the
pairing on $\mathrm{YM}(n)^!$.
For simplicity, we take $(g^{ij})$ to be the identity matrix.
Denote $V=\mathrm{Span}_k\{x_1,x_2,\cdots,x_n\}$,
and
$$
S=\mathrm{Span}_k\left\{
\sum_{i=1}^n[x_i,[x_i,x_j]]:1\le j\le n\right\}\subset V^{\otimes 3}.
$$
Then
$$
\mathrm{YM}(n)=TV/\langle S\rangle.
$$
And
the homogeneous
dual algebra of $\mathrm{YM}(n)$ is given as follows (see Connes and Dubois-Violette \cite[Proposition 1]{CDV} and also
Herscovich and Solotar \cite[Proposition 2.3]{HS2}):
$$
\begin{array}{lll}
\mathrm{YM}(n)^{\vee}_0=k1,& \mathrm{YM}(n)^{\vee}_1=V^*,&\mathrm{YM}(n)^{\vee}_2=\bigoplus_{i,j=1}^n kx_i^*x_j^*,\\
\mathrm{YM}(n)^{\vee}_3=\bigoplus_{i=1}^n kx_i^*z,
& \mathrm{YM}(n)^{\vee}_4=k z^2,&\mathrm{YM}(n)^{\vee}_i=0 \quad(i>4),
\end{array}
$$
where
$z=\sum_{i=1}^n (x_i^*)^2$.
The Koszul dual algebra of $\mathrm{YM}(n)$ is thus given taking
\begin{eqnarray*}
\mathrm{YM}(n)^{!}_0=\mathrm{YM}(n)^{\vee}_0,&&
\mathrm{YM}(n)^{!}_1=\mathrm{YM}(n)^{\vee}_1,\\
\mathrm{YM}(n)^{!}_2=\mathrm{YM}(n)^{\vee}_3,&&
\mathrm{YM}(n)^{!}_3=\mathrm{YM}(n)^{\vee}_4,
\end{eqnarray*}
and the $A_\infty$ operators (recall that we only have two nontrivial
operators $\mu_2$ and $\mu_3$ in this case) are given as follows:
denoting
the above identification by $$\phi:
\mathrm{YM}(n)^{!}_i\stackrel{\cong}{\to} \mathrm{YM}(n)^{\vee}_j,$$
then
for $f_1, f_2, f_3\in\mathrm{YM}(n)^!$,
\begin{equation}
\mu_2(f_1,f_2)=
\left\{
\begin{array}{cl}
\phi^{-1}(\phi(f_1)\cdot\phi(f_2)),
&\mbox{if}\; \phi(f_1)\cdot\phi(f_2)\notin \mathrm{YM}(n)_2^{\vee},\\
0,&\mbox{otherwise,}
\end{array}
\right.
\end{equation}
and
\begin{equation}
\quad\mu_3(f_1,f_2,f_3)=
\left\{
\begin{array}{cl}
\phi^{-1}(\phi(f_1)\cdot\phi(f_2)\cdot\phi(f_3)),&\mbox{if}\;  \phi(f_1)\cdot\phi(f_2)\cdot\phi(f_3)\notin \mathrm{YM}(n)_2^{\vee},\\
0,&\mbox{otherwise.}
\end{array}
\right.
\end{equation}
There is a pairing on $\mathrm{YM}(n)^{\vee}$, which is again
the scaler of the corresponding product with respect to $z^2$,
and hence induces the pairing on $\mathrm{YM}(n)^{!}$ via $\phi$.

\section{Derived representation schemes}\label{Sect_DRep}

As we mentioned in \S\ref{Sect_Intro},
according to the Kontsevich-Rosenberg Principle,
any meaningful non-commutative geometric structure on
an associative algebra $A$ should naturally induce its classical counterpart
on the representative scheme $\mathrm{Rep}_V(A)$,
for all $k$-vector space $V$.
Cuntz and Quillen in \cite{Cuntz_Quillen1} first studied the smoothness problem for associative algebras,
and introduced the notion of {\it smooth algebras}. If an algebra is smooth (e.g. cofibrant),
then its representation scheme is smooth, too. However, in practice,
an algebra is rarely smooth, and therefore it is really difficult to study
the geometry on the representation scheme.

In the following, we first briefly recall the derived representation scheme of a DG algebra,
and then discuss its relations to cyclic homology; most materials are taken from \cite{BFR,BKR}.
After that, we recall several results on
derived non-commutative Poisson structures from \cite{BCER}.

\subsection{Derived representation scheme and cyclic homology}

We start with the representation scheme of associative algebras.
Suppose $A$ is an associative algebra, and $V$ is a $k$-vector space.
Consider the following functor
$$\mathrm{Rep}_V(A): \mathsf{CommAlg}_k\to\mathsf{Sets}, \;
B\mapsto \mathrm{Hom}_{\mathsf{Alg}_k}(A, \mathrm{End}(V)\otimes B),$$
where $\mathsf{CommAlg}_k$ is the category of commutative, unital $k$-algebras
and $\mathsf{Alg}_k$ is the category of unital $k$-algebras.
This functor is representable, which means there exists a commutative algebra, which we denote by
$k[\mathrm{Rep}_V(A)]$, such that
$$\mathrm{Hom}_{\mathsf{Alg}_k}(A, \mathrm{End}(V)\otimes B)
\cong
\mathrm{Hom}_{\mathsf{CommAlg}_k}(k[\mathrm{Rep}_V(A)], B).$$
Keeping $V$ fixed, we in fact get a functor
$$(-)_V: \mathsf{Alg}_k\to\mathsf{CommAlg}_k,\; A\mapsto k[\mathrm{Rep}_V(A)]$$
which we call the {\it representation functor} in $V$,
and the corresponding scheme is called the {\it representation scheme}.
The representation functor can be extended
to the category of DG algebras, $\mathsf{DGA}_k$,
which has a model structure in the sense of Quillen \cite{Quillen}.
In \cite{BKR}
the authors showed that $(-)_V$ defines a left Quillen functor on $\mathsf{DGA}_k$
and therefore it has a total derived functor
$$\mathbf L(-)_V: \mathsf{Ho}(\mathsf{DGA}_k)\to\mathsf{Ho}(\mathsf{CDGA}_k)$$
from the homotopy category of DG algebras to the homotopy category of commutative DG algebras.
When applied to $A$,  $\mathbf L(A)_V$ is called the {\it derived representation scheme} of $A$, and is
represented by a commutative DG algebra, which we denote by
$\mathrm{DRep}_V(A)$.
The homology of $\mathrm{DRep}_V(A)$, namely $\mathrm H_\bullet(\mathrm{DRep}_V(A))$, only depends
on $A$ and $V$, and is called the {\it representation homology} of $A$.
We have that $\mathrm H_0(\mathrm{DRep}_V(A))$ is exactly
$k[\mathrm{Rep}_V(A)]$.

Now let us remind some results in model category theory
(we recommend the excellent survey of Dwyer and Spalinski \cite{DS} for
an introduction to model categories).
Suppose $\mathcal A$ is a model category, then the homotopy category $\mathsf{Ho}(\mathcal A)$
is a category where the objects remain the same as $\mathcal A$,
and the morphisms for two objects, say $A$ and $B$,
are given by
$$\mathrm{Hom}_{\mathsf{Ho}(\mathcal A)}(A,B)
:=\mathrm{Hom}_{\mathcal A}(QA,QB)/\mbox{quasi-equivalences},$$
where $QA$ and $QB$ are the {\it cofibrant resolutions} of $A$ and $B$ respectively
(for more details see \cite[\S5]{DS}).
In the category of non-negatively (or non-positively) graded DG algebras, any quasi-free resolution is a cofibrant resolution.

Now suppose $A$ is a DG algebra, and $R$ is its cofibrant resolution. Then by definition
$\mathbf L(A)_V$ can be represented by
$R\mapsto \mathrm{DRep}_V(A)$.
A key result in \cite{BKR} is the construction of the higher trace map
$$\mathrm{Tr}: \mathrm H_\bullet(R_{\natural})\to \mathrm H_\bullet(\mathrm{DRep}_V(A)).$$
Via Feigin-Tsygan's result
$\mathrm{H}_{\bullet}(\bar R_{\natural})\cong\overline{\mathrm{HC}}_\bullet(A)$
(see \cite{FT} and also \S\ref{Sect_H_C_of_coalg}, Proposition \ref{propresol1})
and by the fact there is a surjective map
$\mathrm{HC}_\bullet(A)\to\overline{\mathrm{HC}}_\bullet(A)$,
we in fact get a map
\begin{equation}
\label{derived_tracemap}
\mathrm{Tr}: \mathrm{HC}_{\bullet}(A)\to\mathrm{H}_\bullet(\mathrm{DRep}_V(A)),
\end{equation}
called the {\it derived trace map},
which, when restricting to degree zero,
is exactly the usual trace map $A_{\natural}=A/[A,A]\to k[\mathrm{Rep}_V(A)]$.

For an associative algebra $A$, $\mathrm{Rep}_V(A)$ has a natural $\mathrm{GL}(V)$ action,
and the covariant space $\mathrm{Rep}_V(A)^{\mathrm{GL}(V)}$ classifies the
isomorphism classes
of representations of $A$ on $V$.
In the derived representation case, there is an analogous result, and moreover,
the authors of \cite{BKR} showed that
there is an isomorphism
$$
\mathrm H_\bullet(\mathrm{DRep}_V(A)^{\mathrm{GL}(V)})\cong\mathrm
H_\bullet(\mathrm{DRep}_V(A))^{\mathrm{GL}(V)}
$$
and the derived trace map \eqref{derived_tracemap} is $\mathrm{GL}(V)$ invariant,
and hence is a map
$$\mathrm{Tr}: \mathrm{HC}_\bullet(A)\to \mathrm H_\bullet(\mathrm{DRep}_V(A))^{\mathrm{GL}(V)}.$$

\subsection{Derived non-commutative Poisson structures}
The above general setting is successfully applied to the study of the derived
non-commutative Poisson structures (\cite{BCER}).

Suppose $A$ is an associative algebra, or more generally a DG algebra.
A {\it derived non-commutative Poisson structure} on $A$,
following the Kontsevich-Rosenberg Principle
in a naive way, is such a structure on its cofibrant resolution $QA$ that naturally
induces a DG Poisson structure on its derived representation scheme
$\mathrm{DRep}_V(A)$ (or more precisely on $\mathrm{DRep}_V(A)^{\mathrm{GL}(V)}$),
for all $V$.
In \cite{BCER}, the authors proposed the following:

\begin{definition}[Derived non-commutative Poisson structure; \cite{BCER}]
Let $A$ be a DG algebra.
A {\it derived non-commutative Poisson bracket} of degree $n$ on $A$
is a DG non-commutative Poisson bracket of degree $n$ in the sense of Crawley-Boevey
on its cofibrant resolution $QA$, namely, it is a differential graded Lie bracket of degree $n$
on $(QA)_\natural$ such that
$$[\bar a,-]:(QA)_\natural\to (QA)_\natural$$
is induced by a derivation $d_a:QA\to QA$, which commutes with the differential.
\end{definition}

The derived non-commutative Poisson bracket
does not depend on the choice of resolutions $QA$ up to homotopy, and hence
is well defined on the homotopy category of DG algebras.
The following is the main result proved in \cite{BCER}:

\begin{theorem}[\cite{BCER} Theorem 9]
Suppose $A$ is an associative (DG) algebra equipped with a derived $n$-Poisson structure.
Then there exists a unique graded $n$-Poisson algebra structure on
$\mathrm H_\bullet(\mathrm{DRep}_V(A))^{\mathrm{GL}(V)}$ for all $V$
such that the derived trace map is a map of graded Lie algebras.
\end{theorem}

\section{Proof of main theorems}\label{Sect_proof_of_main_thm}

\subsection{Proof of Theorems \ref{maintheorem1} and \ref{maintheorem2}}

Let $A$ be an $N$-Koszul $d$-Calabi-Yau algebra. Then by Theorem \ref{Thm_VdB}
there is a cofibrant resolution $p: QA  \stackrel{\simeq}{\twoheadrightarrow}A$, where $QA=\mathbf{\Omega}_\infty(A^{\ac})$
and $A^{\ac}$ is a cyclic $A_\infty$ coalgebra of degree $-d$. Next, by Lemma~\ref{lemmaident} and Corollary~\ref{corimp1},
we get
$$ \mathrm{HC}_{\bullet}(A) \cong \mathrm{HC}_{\bullet}(A^{\ac}) \cong
 \mathrm{H}_{\bullet}[{\bar R_{\natural}}], $$
where $R=\mathbf{\Omega}_\infty(\tilde A^{\ac})$.  And by Corollary~\ref{corimp2}, one has
$$ \mathrm{HH}_{\bullet}(A) \cong \mathrm{H}_{\bullet}[\Omega^1_{R,\natural}] \, .$$
\begin{proof}[Proof of Theorem~\ref{maintheorem1}]
By Theorem~\ref{bracksmain2}, there is a bracket $\{-,-\}$ making $\mathrm{H}_{\bullet}[{\bar R_{\natural}}] $
a $(2-d)$-Lie  algebra and a bracket $\{-,-\}_{\Omega^1_{R,\natural}}$ making
$\mathrm{H}_{\bullet}[\Omega^1_{R,\natural}]$ a Lie module over it.
Since, the first one can be identified with $ \mathrm{HC}_{\bullet}(A)$ and
the second one with $ \mathrm{HH}_{\bullet}(A)$, we have proved this theorem.
\end{proof}
\begin{proof}[Proof of Theorem~\ref{maintheorem2}]
We recall (see  \eqref{KerCoker})  that $\mathrm{CC}_{\bullet}(C)=\mathrm{Ker}(1-T) \cong \mathrm{Coker}(1-T)$.
Hence from the $2$-periodic complex \eqref{bicompbij} one obtains the following exact sequence of complexes
\begin{equation}
\label{exactseqcom1}
 \xymatrix{ 0 \ar[r]& \mathrm{CC}_{\bullet}(C) \ar[r]^i & \Omega^1_{R,\natural} \ar[r]^{-\beta} &  \bar{R} \ar[r]^{\pi} &
 \mathrm{CC}_{\bullet}(C) \ar[r] & 0,}
\end{equation}
since $-\beta=1-T$.  Note we can identify $ \mathrm{CC}_{\bullet}(C)$ with $\bar R_{\natural}$
which is a Lie module and the inclusion $i$ is a DG Lie module homorphism.
$\beta$ is a DG Lie module homorphism by Theorem~\ref{Liemorph1}
and so is the projection $\pi$.

From \eqref{exactseqcom1} we get the Connes long exact sequence for the coalgebra $C$
$$ \xymatrix{ \ar[r] & \mathrm{HH}_{\bullet}(C) \ar[r] & \mathrm{HC}_{\bullet}(C) \ar[r]&
 \mathrm{HC}_{\bullet-2}(C) \ar[r]& \mathrm{HH}_{\bullet-1}(C) \ar[r]&} \,  $$
By the above discussion the maps between homologies are Lie module homomorphisms.
Finally, by Corollary~\ref{corimp1} and Lemma~\ref{lemmaident1},
this sequence can be identified with the Connes
long sequence for the algebra $A$.
\end{proof}

\subsection{Proof of Theorem \ref{maintheorem_DNCP}}
Let us first remind the definition of Hochschild cohomology groups.

\begin{definition}[Hochschild cohomology]\label{Def_HochCo}
Let $A$ be an associative algebra, and $M$ be an $A$-bimodule.
The {\it Hochschild cochain complex} $\mathrm{CH}^{\bullet} (A; M)$ of $A$ with value in $M$
is the complex whose underlying space is
$$
\bigoplus_{n\ge 0}\mathrm{Hom}(A^{\otimes n}, M)
$$
with coboundary $\delta:\mathrm{Hom}(A^{\otimes n}, M)\to \mathrm{Hom}(A^{\otimes n+1}, M)$ defined by
\begin{eqnarray}%
 && (\delta f)(a_0, a_1,a_2,\cdots,a_n) \nonumber\\
 &=& a_0 f(a_1,\cdots,a_n)+
 \sum_{i=0}^{n-1}(-1)^{i+1} f(a_0,\cdots, a_ia_{i+1},\cdots,a_n) 
  +(-1)^{n}f(a_0,\cdots,a_{n-1})a_n.\,\,\,\,\,\,\,\,\,\,\,\label{Hoch_diff}
\end{eqnarray}
The associated cohomology is called the {\it Hochschild cohomology} of $A$ with value in $M$,
and is denoted by $\mathrm{HH}^{\bullet}(A; M)$.
In particular, if $M=A$,
then $\mathrm{HH}^{\bullet}(A; A)$  is called the {\it Hochschild cohomology}
of $A$.
\end{definition}

\begin{definition}
Let $A$ be an associative algebra and let $\mathrm{CH}^{\bullet}(A;A)$
be its Hochschild cochain complex.
\begin{enumerate}
\item[$(1)$]
The {\it Gerstenhaber cup product} on
$\mathrm{CH}^{\bullet}(A;A)$
is defined as follows:
for any $f\in \mathrm{CH}^{n}(A;A)$,
$g \in \mathrm{CH}^{m}(A;A)$,
and $a_1,\ldots,a_{n+m}\in A$,
$$
f\cup g(a_1,\ldots,a_{n+m}):=(-1)^{nm}f(a_1,\ldots,a_{n})g(a_{n+1},\ldots,a_{n+m}).
$$

\item[$(2)$]
For any $f\in \mathrm{CH}^{n}(A;A)$, $g \in \mathrm{CH}^{m}(A;A)$,
and $a_1,\ldots,a_{n+m-1}\in A$,
let
\begin{eqnarray*}
&&
f\circ g (a_1,\ldots,a_{n+m-1})\\
&:=&\sum^{n-1}_{i=1}(-1)^{(m-1)(i-1)}f(a_1,\ldots,a_{i-1},g(a_{i},\ldots,a_{i+m-1}),a_{i+m},\ldots,a_{n+m-1}).
\end{eqnarray*}
The {\it Gerstenhaber bracket} on $\mathrm{CH}^{\bullet}(A;A)$
is defined to be
$$
\{f,g\}_{\mathrm{G}}:=f\circ g-(-1)^{(n-1)(m-1)}g\circ f.
$$
\item[$(3)$]
For any homogeneous elements
$f\in  \mathrm{CH}^{n}(A;A)$ and $\alpha=(a_0, {a}_1, \ldots,  {a}_m)\in  \mathrm{CH}_{m}(A,A)$,
define the \textit{cap product}
$
\cap:  {\mathrm{CH}}_{m}(A;A) \times{\mathrm{CH}}^{n}(A;A) \to {\mathrm{CH}}_{m-n}(A;A)
$
by
$$
\alpha\cap f:=(a_0f( {a}_1,\ldots, {a}_n), {a}_{n+1},\ldots, {a}_m),
$$ for
$m\geq n$,
and zero otherwise.
\end{enumerate}
\end{definition}

Both the Gerstenhaber product and the Gerstenhaber
bracket induce a well-defined product and bracket on Hochschild cohomology
$\mathrm{HH}^{\bullet}(A;A)$, which makes $\mathrm{HH}^{\bullet}(A;A)$
into a Gerstenhaber algebra,
and the cap product makes $\mathrm{HH}_\bullet(A)$ into an $(\mathrm{HH}^{\bullet}(A),\cup)$
module.
Recall that a {\it Gerstenhaber algebra} is
a graded commutative algebra together with a degree $-1$
Lie bracket $\{-,-\}$ such that
$$
a\mapsto\{a,b\}
$$
are derivations with respect to the product.

\begin{theorem}[Gerstenhaber]\label{G-algebra}
Let $A$ be an algebra.
Then the Hochschild cohomology $\mathrm{HH}^{{\bullet}}(A;A)$ of $A$ equipped with the
Gerstenhaber cup product and the Gerstenhaber bracket forms a Gerstenhaber algebra.
\end{theorem}

\begin{proof}
For a proof, see Gerstenhaber \cite[Theorems 3-5]{Gerstenhaber}.
\end{proof}

The Gerstenhaber algebra structure is even more interesting in the case of cyclic $A_\infty$ algebra case.
Tradler, in his paper \cite{Tradler}, showed that for $A_\infty$ algebras,
one may similarly define the Gerstenhaber product and bracket on the Hochschild
cochain complex, where
the Gerstenhaber product is associative up to homotopy, and hence is well-defined
on the cohomology level.
Moreover he showed the following (see \cite[Theorem 2]{Tradler}):

\begin{lemma}[Tradler]
Suppose $C$ is a cyclic $A_\infty$ coalgebra of degree $-d$,
and denote by $A$ its dual $A_\infty$ algebra.
Then there is an isomorphism of graded vector spaces
$$
\Phi:\mathrm{HH}_\bullet(C)\cong\mathrm{HH}^{d-\bullet}(A).
$$
\end{lemma}

\begin{proof}

Since $C$ and hence $A$ are locally finite dimensional,
we have an isomorphism
$$
\begin{array}{cccl}
\phi:&C&\longrightarrow & A[d]\\
&c&\longmapsto&\langle c,-\rangle
\end{array}
$$
of $A$-bimodules,
from which we obtain an isomorphism of chain complexes
$$
\begin{array}{cccl}
\Phi:&\mathrm{CH}_\bullet(C)&\longrightarrow & \mathrm{CH}^{\bullet}(A)\\
&(c_1,c_1,\cdots,c_{n+1})&\longmapsto&\{(a_1,\cdots,a_n)\mapsto (c_1,a_1)\cdots(c_n,a_n)\cdot\phi(c_{n+1})\}.
\end{array}
$$
In other words the isomorphism is given by
$$
\Phi:(c_1,c_2,\cdots, c_{n+1})\longmapsto(c_1,c_2,\cdots, c_{n})\otimes \phi(c_n)
$$
where the right hand side is viewed as an element in $\mathrm{CH}^{n}(A)$
via the isomorphism
\begin{equation*}\mathrm{Hom}(A^{\otimes n}, A)=\big(A^{\otimes n}\big)^{*}\otimes A=
\big(A^*\big)^{\otimes n}\otimes A=C^{\otimes n}\otimes A.\qedhere
\end{equation*}
\end{proof}

\begin{proof}[Proof of Theorem~\ref{maintheorem_DNCP}]
Since $A^{\ac}$ is finite dimensional, from the above lemma
any homogeneous element
$f\in\mathrm{CH}^\bullet(A^!)$
is a linear combination of elements
in the form
$$(u_1,u_2,\cdots,u_n)\otimes \bar u,\quad\mbox{where}\; u_i\in A^{\ac}, \bar u\in A^!.$$
For two elements, say
$f=(u_1,u_2,\cdots,u_n)\otimes \bar u$, $g=(v_1,v_2,\cdots,v_m)\otimes \bar v,$
\begin{eqnarray*}
&&
\{f,g\}_{\mathrm{G}}(x_1,x_2,\cdots,x_{n+m-1})\\
&=&\sum_i (-1)^{\eta_i}f(x_1,\cdots,x_{i-1}, g(x_i,\cdots, x_{i+m-1}),
x_{i+m},\cdots, x_{n+m-1})\\
&-&\sum_j (-1)^{(|f|-1)\cdot(|g|-1)+\eta_j}g(x_1,\cdots, x_{j-1}, f(x_j,\cdots, x_{j+n-1}),x_{j+n},\cdots,
x_{n+m-1})\\
&=&\sum_i(-1)^{\eta_i}
u_1(x_1)\cdots u_{i-1}(x_{i-1})v_1(x_i)\cdots v_m(x_{i+m-1})u_i(\bar v) u_{i+1}(x_{i+m})\cdots u_n (x_{n+m-1})
\bar u\\
&-&\sum_j(-1)^{\eta+\eta_j}
v_1(x_1)\cdots v_{j-1}(x_{j-1})u_1(x_j)\cdots u_n(x_{j+n-1})v_j(\bar u)v_{j+1}(x_{j+n})\cdots v_m(x_{n+m-1})
\bar v,
\end{eqnarray*}
for any homogeneous element $(x_1,x_2,\cdots, x_{n+m-1})\in (A^{!})^{\otimes n+m-1}$,
where $\eta_i=|g|(|x_1|+\cdots+|x_{i-1}|)$, $\eta_j=|f|(|x_1|+\cdots+|x_{j-1}|)$ and $\eta=(|f|-1)(|g|-1)$.
In other words,
\begin{eqnarray}\label{Gersten_br}
\{f,g\}_{\mathrm{G}}&=&\sum_i(-1)^{\eta_i}
u_i(\bar v)(u_1,\cdots, u_{i-1}, v_1,\cdots, v_m, u_{i+1},\cdots, u_n)\otimes \bar u\nonumber\\
&-&\sum_j (-1)^{\eta+\eta_j}v_j(\bar u)(v_1,\cdots, v_{j-1}, u_1,\cdots, u_n, v_{j+1},\cdots, v_m)\otimes \bar v.
\end{eqnarray}

Now let $u=[u_1|u_2|\cdots|u_n]$, $v=[v_1|v_2|\cdots|v_m]\in\mathbf\Omega_\infty(A^{\ac})$ be two homogeneous elements where
$u_i$'s and $v_j$'s are generically different from each other.
Viewing them as elements in $\mathrm{CH}_\bullet(A^{\ac})$ via the embedding (\ref{CobartoHoch})
or (\ref{Cobar_to_Hoch_Ainfty}),
we have
\begin{equation*}
B(u)=\sum_{i=1}^n(-1)^{\varepsilon_i} (u_{i+1}, \cdots, u_n, u_1,\cdots, u_i), \quad
B(v)=\sum_{j=1}^m (-1)^{\varepsilon_j}(v_{j+1}, \cdots, v_m, v_1,\cdots, v_j),
\end{equation*}
where $\varepsilon_i=(|u_1|+\cdots+|u_i|)(|u_{i+1}|+\cdots+|u_n|)$ and
$\varepsilon_j=(|v_1|+\cdots+|v_j|)(|v_{j+1}|+\cdots+|v_m|)$.
Applying the above lemma to $A^{\ac}$, we have
\begin{eqnarray}
\Phi\circ B(u)&=&\sum_{i=1}^n(-1)^{\varepsilon_i} (u_{i+1}, \cdots, u_n, u_1,\cdots, u_{i-1})\otimes \bar u_i,
\label{PhiB(u)}\\
\Phi\circ B(v)&=&\sum_{j=1}^m (-1)^{\varepsilon_j}(v_{j+1}, \cdots, v_m, v_1,\cdots, v_{j-1})\otimes \bar v_j,
\label{PhiB(v)}
\end{eqnarray}
where $\bar u_i=\phi(u_i)$ and $\bar v_j=\phi(v_j)$.
Thus for two summands in $\Phi\circ B(u)$ and $\Phi\circ B(v)$ respectively,
say
$$(u_1,\cdots, u_{n-1})\otimes\bar u_n,\quad
(v_1,\cdots, v_{m-1})\otimes \bar v_m,$$
by (\ref{Gersten_br}) we have
\begin{eqnarray}\label{Gersten_bracket}
&&
\{(u_1,\cdots, u_{n-1})\otimes\bar u_n,(v_1,\cdots, v_{m-1})\otimes \bar v_m\}_{\mathrm{G}}\\
&=&\sum_i(-1)^{\mu_i}
\langle sv_m,su_i\rangle (u_1,\cdots, u_{i-1}, v_1,\cdots, v_{m-1}, u_{i+1},\cdots, u_{n-1})\otimes \bar u_n\nonumber\\
&-&\sum_j (-1)^{\nu_j}\langle su_n,sv_j\rangle
(v_1,\cdots, v_{j-1}, u_1,\cdots, u_{n-1}, v_{j+1},\cdots, v_{m-1})\otimes \bar v_m,\nonumber
\end{eqnarray}
where $\mu_i=|v_m|+(|v|+d)(|u_{i+1}|+\cdots+|u_n|+d)$,
$\nu_j=|u_n|+(|u|+d)(|v_1|+\cdots+|v_{j-1}|)$,
and $\langle -,-\rangle$ is the graded symmetric pairing on $C$.

On the other hand, by (\ref{def_bb}) we have
\begin{eqnarray}
&&\Phi\circ B\{(u_1,\cdots u_n),(v_1,\cdots, v_m)\}_{\mathrm{DNCP}}\nonumber\\
&=&\Phi\circ B
\Big(\sum_{i,j}(-1)^{|u_i|+\epsilon}\langle su_i, sv_j\rangle
\cdot(v_1,\cdots, v_{j-1}, u_{i+1},\cdots, u_n,u_1,\cdots, u_{i-1}, v_{j+1},\cdots, v_m)\Big)
\,\,\,\,\,\,\,\,\,\,\,\,\, \label{DNCP_bracket}
\end{eqnarray}
where $\epsilon$ is given as in (\ref{def_bb}).
From this one sees that
(\ref{Gersten_bracket}) are summands in (\ref{DNCP_bracket}).
For other summands in $\Phi\circ B(u)$ and $\Phi\circ B(v)$, by the same argument,
their Gerstenhaber brackets also lie in (\ref{DNCP_bracket}) as summands.

Also, from (\ref{Gersten_bracket}) one sees that
in the expression
$$
\{\Phi\circ B(u),\Phi\circ B(v)\}_{\mathrm{G}}
$$
none of the summands literally cancels with any other,
and therefore to show
$$
\{\Phi\circ B(u),\Phi\circ B(v)\}_{\mathrm{G}}=\Phi\circ B\{u,v\}_{\mathrm{DNCP}},$$
it is enough to show both sides have the same number of summands.

In fact, on one hand, from (\ref{PhiB(u)}) and (\ref{PhiB(v)}) we see that
$\Phi\circ B(u)$ and $\Phi\circ B(v)$ have $n$ and $m$ summands respectively,
and therefore
$$
\{\Phi\circ B(u),\Phi\circ B(v)\}_{\mathrm{G}}
$$
has $mn(m+n-2)$ summands.
On the other hand,
from (\ref{DNCP_bracket}) we know that
$\{u,v\}_{\mathrm{DNCP}}$ has $mn$ summands, and each summand contains $m+n-2$  components, and
therefore
$\Phi\circ B\{u,v\}_{\mathrm{DNCP}}$ has $mn(m+n-2)$ summands.
The numbers of summands are equal, which proves the theorem.
\end{proof}

Let us now recall a theorem of Keller \cite{Keller0} (see also Herscovich \cite{Herscovich}):

\begin{theorem}[Keller]\label{thm_keller}
Let $A$ be an $N$-Koszul algebra, and denote by $A^{!}$ its Koszul dual algebra.
Then we have isomorphism
$$
\Xi:\mathrm{HH}^\bullet(A^!)\to\mathrm{HH}^\bullet(A)
$$
of Gerstenhaber algebras.
\end{theorem}

\begin{proof}
For $N=2$, this is proved in Keller \cite[Theorem 3.5]{Keller0}.
For the general case, it is proved by Herscovich in \cite{Herscovich}.
\end{proof}

\begin{corollary}
\label{maincor1}
Let $A$ be an $N$-Koszul Calabi-Yau algebra.
Then $\Xi\circ\Phi\circ B$ maps
the Leibniz-Loday bracket $\{-,-\}_{\mathrm{DNCP}}$ on $\tilde R=\mathbf\Omega_{\infty}(\tilde A^{\ac})$ to the Gerstenhaber
bracket $\{-,-\}_{\mathrm G}$ on $\mathrm{HH}^\bullet(A)$.
\end{corollary}

\begin{proof}
The claim follows from a combination of Theorem~\ref{maintheorem_DNCP} and
Theorem \ref{thm_keller} of Keller above.
\end{proof}

Roughly speaking, the above corollary says that the derived non-commutative Poisson bracket
of a Koszul Calabi-Yau algebra is naturally mapped to the Gerstenhaber bracket on its Hochschild cohomology.

\subsection{Proof of Theorem \ref{maintheorem_dVV}}\label{proof_of_maintheorem_dVV}

Before we proceed to the proof of Theorem~\ref{maintheorem_dVV},
let us recall a Batalin-Vilkovisky algebra structure on the Hochshcild cohomology
$\mathrm{HH}^{\bullet}(A)$.
In order to do that, let us first observe that, for an associative algebra $A$, we have the following two structures:
\begin{enumerate}
\item[$-$] there is the Connes cyclic operator
$B: \mathrm{HH}_{\bullet}(A) \to \mathrm{HH}_{\bullet+1}(A)$ whose square is zero; and
\item[$-$] there is a Gerstenhaber algebra structure on $(\mathrm{HH}^{\bullet}(A), \cup)$,
and $\mathrm{HH}_\bullet(A)$ is a module over $\mathrm{HH}^{\bullet}(A)$. Denote the action
by $\cap$.
\end{enumerate}
For a $d$-Calabi-Yau algebra $A$, these two structures are perfectly matched.
Note that there is an element $\omega\in\mathrm{HH}_d(A)$ such that the following map
\begin{equation}\label{PD_isomorphism}
\begin{array}{cccl}
\Psi:&\mathrm{HH}^{\bullet}(A)& \stackrel{\simeq}\longrightarrow & \mathrm{HH}_{d-\bullet}(A)\\
&f&\longmapsto& \iota_{f}\omega
\end{array}
\end{equation}
is an isomorphism,
called the {\it non-commutative Van den Bergh-Poincar\'e duality};
for a proof of this result see \cite[Proposition 5.5]{dTdVVdB}.

\begin{definition}[Batalin-Vilkovisky algebra]
Let $(V, \hdot)$ be a graded commutative algebra.
A {\it Batalin-Vilkovisky operator} on $V$
is a degree $-1$ differential operator
$$
\Delta:V_{i}\to V_{i-1}, \quad\mbox{for all}\; i
$$
such that the {\it deviation from being a derivation}
\begin{equation}\label{deviation}
\{a, b\}:=(-1)^{|a|}(\Delta(a \hdot b) - \Delta(a) \hdot b-(-1)^{|a|} a\hdot \Delta(b)),
\quad\mbox{for all}\,\, a, b\in V
\end{equation}
is a derivation for each component, {\it i.e.}
$$
\{a, b \hdot c\}=\{a,b\} \hdot c+(-1)^{|b|(|a|-1)}b \hdot \{a, c\},\quad\mbox{for all}\,\, a, b,c\in V.
$$
The triple $(V, \hdot, \Delta)$ is called a {\it Batalin-Vilkoviksy algebra}.
\end{definition}

If $(V, \hdot, \Delta)$ is a Batalin-Vilkoviksy algebra,
then by definition $(V, \hdot, \{-,-\})$, where $\{-,-\}$ is given by (\ref{deviation}),
is a Gerstenhaber algebra. In other words, a Batalin-Vilkovisky algebra is a special
class of Gerstenhaber algebras.

Let $A$ be a $d$-Calabi-Yau algebra and let
$$
\Delta:= \Psi^{-1} \circ B \circ \Psi:
\mathrm{HH}^{\bullet}(A) \to \mathrm{HH}^{\bullet-1}(A)  \, ,
$$
we have
\begin{theorem}[Ginzburg \cite{Ginzburg}, Theorem 3.4.3]\label{BV_CY}
Suppose that $A$ is an $d$-Calabi-Yau algebra.
Then $(\mathrm{HH}^{\bullet}(A;A), \cup , \Delta) $ is a Batalin-Vilkovisky algebra.
Moreover, the bracket \eqref{deviation} for this triple
is exactly the classical Gerstenhaber bracket $\{-,-\}_{\mathrm G}$.
\end{theorem}

Recently the first author with the fourth author and G. D. Zhou
proved in \cite[Theorem A]{CYZ} that for Koszul (and more generally, $N$-Koszul) Calabi-Yau algebras,
the isomorphism in Theorem \ref{thm_keller} is in fact an isomorphism of
Batalin-Vilkovisky algebras. The key point in the proof
is that there is a canonical complex (the Koszul complex equipped with an appropriate differential)
which computes both $\mathrm{HH}_\bullet(A)$ and $\mathrm{HH}_\bullet(A^{\ac})$, and we have
the following commutative diagram
\begin{equation}\label{diag_BViso}
\xymatrixcolsep{4pc}
\xymatrix{
\mathrm{HH}^\bullet(A^!)\ar[r]^{\Xi}&\mathrm{HH}^{\bullet}(A)\ar[d]^{\Psi}\\
\mathrm{HH}_{d-\bullet}(A^{\ac})\ar@{=}[r]\ar[u]_{\Phi}&\mathrm{HH}_{d-\bullet}(A).
}
\end{equation}

\begin{proof}[Proof of Theorem~\ref{maintheorem_dVV}]
Let $\alpha, \beta \in \mathrm{HC}_{\bullet}(A)$. Then
\begin{eqnarray}\label{eqmth4}
\;\;\;\;\;\;\;\;\;\;\;\;\;\;
 \big\{B(\alpha), B(\beta)\big\}_{\mathrm{dVV}}
 &=& (-1)^{|\alpha|+d+1} B \circ \Psi \big(  \Psi^{-1} ( \pi( B(\alpha))  \cup \Psi^{-1}( \pi(B(\beta)) \big)  \\[0.1cm]
 &=& (-1)^{|\alpha|+d+1} \Psi \circ \Delta \big(  \Psi^{-1} ( \pi( B(\alpha))  \cup \Psi^{-1}( \pi(B(\beta)) \big) \nonumber \\[0.1cm]
 &=& (-1)^{|\alpha|+d+1} \Psi \circ \Delta \big(  \Delta( \Psi^{-1}(\alpha)) \cup  \Delta(\Psi^{-1}(\beta))\big)  \nonumber \\[0.1cm]
 &=&  (-1)^{|\alpha|+d+1+d-|\alpha|-1}\Psi \big\{ \Delta( \Psi^{-1}(\alpha)) ,  \Delta(\Psi^{-1}(\beta))\big\}_{\mathrm G} \nonumber\\[0.1cm]
 &=& \Psi \big\{ \Psi^{-1}(B(\alpha)), \Psi^{-1}(B(\beta)) \big\}_{\mathrm G} \nonumber\\[0.1cm]
 &=& \Psi \Xi \big\{ \Xi^{-1} (\Psi^{-1}(B(\alpha))),  \Xi^{-1}(\Psi^{-1}(B(\beta)))\big\}_{\mathrm G} \nonumber\\[0.1cm]
 &=& \Phi^{-1} \big\{\Phi\circ B(\alpha), \Phi\circ B(\beta)\big\}_{\mathrm{G}} \nonumber\\[0.1cm]
 &=& B \big\{\alpha, \beta\big\}_{\mathrm{DNCP}} ,\nonumber
 \end{eqnarray}
where the first two equalities follow definitions of $\{-,-\}_{\mathrm{dVV}}$ and $\Delta$ respectively,
the third one follows from \eqref{comdiag1}
and the definition of $\Delta$, the fourth equality follows from applying
\eqref{deviation},
the fifth equality follows again from the definition of $\Delta$,
the sixth equality holds due to Theorem \ref{thm_keller},
the seventh equality follows from \eqref{diag_BViso},
and finally the eighth equality follows from Theorem~\ref{maintheorem_DNCP}.
\end{proof}
\begin{proof}[Proof of Corollary~\ref{corexpbracket}]
It follows from \eqref{eqmth4} that
$$
B \{\alpha, \beta\}_{\mathrm{DNCP}} = (-1)^{|\alpha|+d+1}B \circ \Psi \big( \Psi^{-1} ( B(\alpha))  \cup \Psi^{-1} (B(\beta))\big).
$$
Since $A$ is a graded algebra the Connes operator $B$ is injective (see e.g. \cite[Theorem 9.9.1]{Weibel})
and hence
$$
\{\alpha, \beta\}_{\mathrm{DNCP}} =  (-1)^{|\alpha|+d+1}\Psi \big(  \Psi^{-1} ( B(\alpha))  \cup \Psi^{-1} (B(\beta))\big).
$$
Note that by \eqref{PD_isomorphism} the RHS of the above identity
\begin{eqnarray*}
\Psi \big( \Psi^{-1} ( B(\alpha))  \cup \Psi^{-1} (B(\beta))\big)
&=& \iota_{\Psi^{-1} ( B(\alpha))  \cup \Psi^{-1} (B(\beta))} \omega\\
&=& \iota_{\Psi^{-1} (B(\alpha))} \iota_{\Psi^{-1} ( B(\beta))}\omega\\
&=&  \iota_{\Psi^{-1} ( B(\alpha))} B(\beta),
\end{eqnarray*}
and the result follows.
\end{proof}

\section{Example: the polynomial algebra}\label{Sect_example_poly_alg}

Let $A$  be the symmetric algebra  $\mathrm{Sym}(V)$  where $V$ is a
vector space of dimension $m \ge 1$ concentrated in homological degree $0$.
This is a quadratic Koszul, Calabi-Yau algebra of dimension $m$.
By the well-known result of Hochschild-Kostant-Rosenberg, we
can identify $\mathrm{HH}_{\bullet}(A) \cong \Omega^{\bullet}(V)$
and $\mathrm{HH}^{\bullet}(A) \cong \Theta_{\bullet}(V)$,
where $\Omega^{\bullet}(V)$ is the  vector spaces of differential forms
and $ \Theta_{\bullet}(V)$ is that  of polyvector  fields on $V$.
Moreover, $\mathrm{HC}_{\bullet}(A)$ can be  identified
with $\oplus_{n} \Omega^{n}(V) / d(\Omega^{n-1}(V))$, where
$d$ is the de Rham differential.
Using Corollary~\ref{maincor1}
we will  give an geometric description of $\{-,-\}_{\mathrm{DNCP}}$ (to simplify the notation
in the following we write $\{-,-\}_{\mathrm{DNCP}}$ as $\{-,-\}$).

\subsection{Koszul dual coalgebra}

We recall that $A$ is a quadratic Koszul algebra defined by the quadratic data $(V, S)$
with $ S \subset V \otimes V$ spanned by the vectors of the form
$ v \otimes u - u \otimes v$, which has a minimal semi-free resolution
$R:=\mathbf{\Omega}(C)$ given by the cobar construction of the Koszul dual coalgebra
$C:=A^{\ac}=C(sV,s^2S)$. The coalgebra $C$ is an exterior coalgebra
whose degree $n$ elements we denote by
$$ \mu (v_1,v_2, \cdots, v_n):=(sv_1 \wedge sv_2 \wedge \cdots \wedge sv_n) \, .$$
We can always assume that $v_1,\cdots,v_n$ are linearly independent.
The coproduct $\triangle : C \to C \otimes C$ is given by
$$ \triangle ( \mu (v_1,v_2, \cdots, v_n)) := \sum_{p=0}^n \,  \sum_{\sigma \in Sh_{p,n-p}}
\mathrm{sgn}(\sigma) \, \mu(v_{\sigma(1)}, \cdots , v_{\sigma(p)}) \otimes
 \mu(v_{\sigma(p+1)}, \cdots , v_{\sigma(n)}),$$
where $Sh_{p,n-p} \subset S_n$ is  the subset of $(p,n-p)$ shuffles. We denote counit
in $C$ by $e$. For example,  we have
 $$
 \triangle ( \mu (v_1,v_2)) = e \otimes \mu (v_1,v_2) + \mu (v_1,v_2)   \otimes e
 + \mu(v_1)  \otimes \mu(v_2) -  \mu(v_2)  \otimes \mu(v_1) \, .
 $$
We recall that   a symmetric bilinear form of degree $-d$ on  $C$ is a
pairing $ \langle \,  , \,  \rangle : C \otimes C \to k[d]$  such that
\begin{equation}
\label{bilform}
   \langle a, b \rangle =  (-1)^{|a||b|} \langle b, a \rangle \quad \mbox{for all} \, a,b \in C \ .
\end{equation}
It is \textit{cyclic} if
\begin{equation}
\label{def_cocyc}
\langle a, b^2 \rangle\cdot b^1 \, = \,  \langle a^1, b \rangle\cdot a^2, \quad\mbox{for all}\, a,b\in C \, ,
\end{equation}
where $\triangle(a)=\sum\, a^1 \otimes a^2 $ and $\triangle (b) = \sum\, b^1 \otimes b^2$ using Sweedler's notation.

\subsection{Cyclic forms on $C$}

First we show
\begin{proposition}
\label{prop1}
Let  $\langle -, - \rangle$  be a cyclic form on $C$. Then
\begin{enumerate}
\item[$(1)$]
$\langle e , \mu (v_1,v_2, \cdots, v_n) \rangle = 0 $ for $n \le m-1$.
\item[$(2)$]
$\langle  \mu (v_1,v_2, \cdots, v_n),  \mu (w_1,  \cdots, w_l) \rangle = 0 $  for $l \le m-1$  if each $v_i \in \Sp \{ w_1,\cdots,w_l\} $.
\end{enumerate}
\end{proposition}
\begin{proof}
(1)
Since $n \le m-1$ we can choose $v \in V$ such that $v$ is not in  $\Sp \{ v_1,\cdots,v_n\}$.
Let $a=\mu(v)$  and $b= \mu (v_1,v_2, \cdots, v_k)$. Then
$$ \triangle (a) = e\otimes \mu(v)+ \mu(v) \otimes e \, ,
\triangle (b) = e\otimes \mu(v_1,\cdots,v_n)+ \mu(v_1,\cdots,v_n) \otimes e +\cdots \, .$$
Using \eqref{def_cocyc} we get
\begin{eqnarray}
&&  \langle \mu(v),
\mu (v_1, \cdots, v_n) \rangle\cdot e +  \langle \mu(v) ,e\rangle\cdot \mu (v_1, \cdots, v_n)\, +\cdots\nonumber \\
& =& \langle \mu(v) ,  \mu (v_1,\cdots, v_n) \rangle\cdot e + \langle e,  \mu (v_1,\cdots, v_n)\rangle\cdot \mu(v).  \nonumber
\end{eqnarray}
Since the left  hand side does not contain the term with $\mu(v)$ we can conclude this statement.

(2)
Due to linearity of the form it suffices to show it in the case when each $v_i$ is of one $w_j$'s and
after reordering we can also assume that $v_1=w_1,\cdots,v_n=w_n$. So we only need to show
$$\langle  \mu (w_1,w_2,\cdots, w_n),  \mu (w_1,\cdots, w_l) \rangle = 0 \, . $$
First, we choose $v \in V$ such that $v \notin \Sp \{ w_1,\cdots,w_l\}$.
Second, let $a= \mu (w_1, \cdots, w_n)$ and $b=\mu (w_1, \cdots, w_l,v)$. Then
\begin{eqnarray}
&& \triangle (a) = e\otimes \mu(w_1,\cdots,w_n)+ \mu(w_1,\cdots,w_n) \otimes e +\cdots \, ,   \nonumber \\
&& \triangle (b) = \cdots+  \mu(w_1,\cdots,w_l) \otimes
 \mu(v) +  (-1)^{ l} \mu(v) \otimes \mu(w_1,\cdots,w_l)+ \cdots \, . \nonumber
\end{eqnarray}
Now using  \eqref{def_cocyc} we get
\begin{eqnarray}
&&  (-1)^{l} \langle \mu(w_1,\cdots,w_n),\mu(v) \rangle\cdot  \mu (w_1,\cdots, w_l)+
\langle \mu(w_1,\cdots,w_n),
\mu (w_1, \cdots, w_l) \rangle\cdot \mu(v) + \cdots  \nonumber  \\
&=& \langle e,  \mu (w_1,\cdots, w_l,v)\rangle\cdot \mu(w_1,\cdots,w_n) + \langle \mu(w_1,\cdots,w_n) ,  \mu (w_1,\cdots, w_l,v) \rangle\cdot e + \cdots
 \nonumber
\end{eqnarray}
Left side has no term with $\mu(v)$ while right side has only one term with $\mu(v)$
and therefore coefficient in front of this
term must vanish.
\end{proof}

Next, we prove the following statement
\begin{proposition}
\label{prop2}
If $ v \notin  \Sp \{v_1,\cdots,v_n, w_1,\cdots,w_l\}$ then
$$\langle  \mu ( v_1,v_2, \cdots, v_n, v),  \mu (w_1,\cdots, w_l) \rangle
=\langle  \mu ( v_1,v_2,\cdots, v_n),  \mu (v, w_1, \cdots, w_l) \rangle \,. $$
\end{proposition}
\begin{proof}
For $a= \mu ( v_1,v_2, \cdots, v_n, v)$ and $b= \mu (v, w_1, \cdots, w_l)$:
\begin{eqnarray}
&&  \triangle (a) =   \mu ( v_1, \cdots, v_n) \otimes \mu(v) + (-1)^n \mu(v) \otimes  \mu ( v_1,\cdots, v_n)+\cdots  \nonumber \\
&&  \triangle (b) = \mu(v) \otimes \mu (w_1, \cdots, w_l) + (-1)^l \mu (w_1,\cdots, w_l)\otimes \mu(v) +\cdots \nonumber
\end{eqnarray}
Since $v\notin \Sp \{ v_1,\cdots,v_n, w_1,\cdots,w_l\}$,  the above terms in coproducts of $a$ and $b$ are
the only terms containing $\mu(v)$.
By  \eqref{def_cocyc} we have
\begin{equation}
\langle  \mu ( v_1,v_2, \cdots, v_n, v),  \mu (w_1,  \cdots, w_l) \rangle \mu(v) +\cdots =
\langle  \mu ( v_1,v_2, \cdots, v_n),  \mu (v, w_1,  \cdots, w_l) \rangle  \mu (v) + \cdots  \, \nonumber
\end{equation}
This finishes our proof.
\end{proof}

The next corollary follows immediately from Propositions \ref{prop1} and \ref{prop2}.

\begin{corollary}
\label{cor1}
We have
$$  \langle  \mu (v_1,v_2, \cdots, v_n),  \mu (w_1, \cdots, w_l) \rangle = 0  $$
unless $W:=\Sp \{ v_1,v_2,\cdots, v_n,w_1,\cdots,w_l\} =V$.
\end{corollary}

\begin{proof}
Let $ W \neq V$. First, we consider the case when none of $v_i$'s in $W$.
 Then using Proposition~\ref{prop2}
we obtain
$$  \langle  \mu (v_1,v_2,\cdots, v_n),  \mu (w_1,\cdots, w_l) \rangle \, = \, \langle  e,
 \mu(v_1,v_2,\cdots, v_n, w_1,\cdots,w_l) \rangle  \, .$$
The RHS is $0$ by Proposition~\ref{prop1} part $(1)$.

Second, we can assume that $v_i \in W$ for all $i=1,\cdots,n$. Otherwise, we can repeatedly
use Proposition~\ref{prop2} to move those $v_i$'s which are not in $W$ to the right.
Finally, by Proposition~\ref{prop1} part $(2)$, we have vanishing of the bilinear form in this case.
\end{proof}
The above results can be summarized into the following statement
\begin{proposition}
Let $V$ be an $m$-dimensional vector space and let $A:=\mathrm{Sym}(V)$. Then  the Koszul
dual coalgebra $C$ has a unique cyclic form and it is of degree $-m$. This form
is completely determined by its value $ \langle e,  \mu(v_1,v_2,\cdots,v_m) \rangle$, where
$ \{v_1,v_2,\cdots,v_m\}$ is a basis for $V$.
\end{proposition}

\subsection{Lie bracket on $\mathrm{HC}_{\bullet}(A)$}

We have an isomorphism $\Psi \, : \, \Theta_{k} (V) \, \rightarrow \,\Omega^{m-k} (V)$
such that $\Psi (\xi) = \iota_{\xi} \omega$.
The space $\Theta_* (V)$ equipped with BV operator $\Delta : \Theta _{*} (V) \rightarrow \Theta_{*-1} (V)$
such that the following diagram commutes:
$$
\xymatrixcolsep{4pc}
\xymatrix{
\Theta_{k} (V) \,  \ar[d]^{\Psi} \ar[r]^{\Delta} & \, \Theta_{k-1} (V) \ar[d]^{\Psi} \\
\Omega^{n-k} (V) \, \ar[r]^{d} &\, \Omega^{m-k+1} (V)}
$$
By Corollary~\ref{maincor1},  the bracket $\{- ,- \}$ on $\Omega^*(V)$ is given by
$$
d \{ \alpha, \beta\} \, = \,  \Psi \, \{ \Psi^{-1} (d \alpha) , \Psi^{-1} (d \beta)\}_{\mathrm G}
$$
where $\alpha, \beta \in \Omega^*(V)$ and $\{ -,-\}_{\mathrm G}$ is the Gerstenbaher bracket of poly-vector fields.

\begin{lemma}
For given forms $\alpha$ and $\beta$, their bracket is given by
$$
d \, \{\alpha , \beta\} = (-1)^{m-|\alpha|-1}d \, \iota_{\xi} d \beta = (-1)^{(m-|\alpha|-1)(m-|\beta|)}d \, \iota_{\eta} d \alpha
$$
where $\Psi^{-1} (d\alpha) = \xi$ and $\Psi^{-1} (d \beta)= \eta$.
\end{lemma}

\begin{proof}
Recall for $a, b \in \Theta_* (V)$ their Gerstenbaher bracket can expressed in terms of $\Delta$
$$
\{a,b\}_{\mathrm{G}} = \, (-1)^{|a|} \big( \Delta(a \wedge b) - \Delta(a) \wedge b - (-1)^{|a|} a \wedge \Delta(b) \big).
$$
Hence using commutativity above diagram we have
$$
\{ \Psi^{-1} (d \alpha) , \Psi^{-1} (d \beta)\}_{\mathrm{G}}
= \, (-1)^{m-|\alpha|-1} \Delta ( \Psi^{-1}(d \alpha) \wedge \Psi^{-1}(d \beta) ),
$$
and applying $\Psi$ we have
\begin{eqnarray*}
  \Psi \{ \Psi^{-1} (d \alpha) , \Psi^{-1} (d \beta)\}_{\mathrm{G}}
  &=& (-1)^{m-|\alpha|-1} \Psi \, \Delta \big( \Psi^{-1}(d \alpha) \wedge \Psi^{-1}(d \beta) \big) \\
  &=& (-1)^{m-|\alpha|-1} d \circ \Psi \big( \Psi^{-1}(d \alpha) \wedge \Psi^{-1}(d \beta) \big).
\end{eqnarray*}
Now using well-known formula
$$
\iota_{\xi \wedge \eta} \omega = \, \iota_{\xi} \, ( \iota_{\eta} \omega) =(-1)^{|\xi||\eta|}\iota_{\eta\wedge \xi} \omega
$$
we obtain
$$
\Psi \big( \Psi^{-1}(d \alpha) \wedge \Psi^{-1}(d \beta) \big)
= \, \iota_{\Psi^{-1}(d \alpha) \wedge \Psi^{-1}(d \beta)} \, \omega
= \, \iota_{\Psi^{-1}(d \alpha) } \, (\iota_{\Psi^{-1}(d \beta)} \, \omega)
= \, \iota_{\Psi^{-1}(d \alpha) } \, d \beta ,
$$
from which the result immediately follows.
\end{proof}

Since $\mathrm{HH}_{\bullet}(A) \cong \Omega^{\bullet}(V)$ and
 $\mathrm{HC}_{\bullet}(A) \cong \oplus_n \Omega^{n}(V)/ d(\Omega^{n-1}(V))$, we have established
\begin{corollary}
The Lie bracket on $\mathrm{HC}_{\bullet}(A)$ and the Lie module
structure on $\mathrm{HH}_{\bullet}(A)$
in Theorem~\ref{maintheorem1} for $A=\mathrm{Sym}(V)$
is given as follows. For $\alpha, \beta \in \mathrm{HC}_{\bullet}(A)$  or
$\beta \in  \mathrm{HH}_{\bullet}(A)$,
we have
$$  \{ \alpha, \beta\}=(-1)^{(m-|\alpha|-1)(m-|\beta|)}\iota_{\eta} d \alpha,
$$
where  $\eta=\Psi^{-1} (d \beta)$.
\end{corollary}

\begin{remark}
One can directly check that $\{- ,- \}$ is not a Lie bracket on $\Omega^*(V)$, since the Jacobi identity
fails.  Thus $\{- ,- \}$ defines only a Lie module structure on $\mathrm{HH}_{\bullet}(A)$ not Lie algebra.

More precisely, let
$$
\alpha = x^2 yz dx, \quad \beta = xyz dy, \quad \gamma = xz dz,
$$
and $\omega=dx\wedge dy\wedge dz$,
then
$$
\Psi^{-1}(d \alpha) = -x^2 z \frac{\partial}{\partial z} + x^2y  \frac{\partial}{\partial y},
 \quad \Psi^{-1}(d\beta) = yz \frac{\partial}{\partial z} - xy \frac{\partial}{\partial x},  \quad
 \Psi^{-1}(d\gamma) = -z \frac{\partial}{\partial y} ,
$$
and we have
\begin{eqnarray*}
 \{ \{\alpha, \beta\} , \gamma\} &=& x^2yz^2 dx - x^3yz dz,\\
 \{\alpha, \{\beta, \gamma\} \} &=& -x^3 z^2 dy - x^3 yz dz,\\
 \{\beta, \{\alpha, \gamma\}  \} &=& 2x^2yz^2 dx + 2x^3yzdz.
\end{eqnarray*}
Thus we have
$$
\{\{\alpha, \beta\} , \gamma\} - \{\alpha, \{\beta, \gamma\}\}
+ \{\beta, \{\alpha, \gamma\}\}
= 3x^2yz^2 dx +x^3 z^2 dy+ 2x^3yzdz = d(x^3 y z^2)\,,
$$
which is not zero, and therefore $\mathrm{HH}_{\bullet}(A)$ with $\{-,-\}$
above does not form a Lie algebra.
\end{remark}



\end{document}